\documentclass{amsart}
\usepackage{shortsalch,tikz,subcaption,float}
\usepackage{cleveref}
\usepackage{xy}
\usetikzlibrary{calc}
\xyoption{all}
\title{K\"{u}nneth formulas for Cotor.}
\author{A. Salch}
\begin{document}
\maketitle
\begin{abstract}
We investigate the question of how to compute the cotensor product, and more generally the derived cotensor (i.e., Cotor) groups, of a tensor product of comodules. In particular, we determine the conditions under which there is a K\"{u}nneth formula for Cotor. We show that there is a simple K\"{u}nneth theorem for Cotor groups if and only if an appropriate coefficient comodule has trivial coaction. This result is an application of a spectral sequence we construct for computing Cotor of a tensor product of comodules. Finally, for certain families of nontrivial comodules which are especially topologically natural, we work out necessary and sufficient conditions for the existence of a K\"{u}nneth formula for the $0$th Cotor group, i.e., the cotensor product. We give topological applications in the form of consequences for the $E_2$-term of the Adams spectral sequence of a smash product of spectra, and the Hurewicz image of a smash product of spectra.
\end{abstract}
\tableofcontents

\section{Introduction.}

The classical, well-known K\"{u}nneth formula expresses 
\begin{itemize}
\item the homology of a tensor product of chain complexes of abelian groups\footnote{Or, more generally, of chain complexes of modules over any commutative ring $R$, if one is willing to consider the K\"{u}nneth spectral sequence.} in terms of the (derived) tensor product of the homology of each chain complex,
\item and, closely related, the homology of a Cartesian product of topological spaces in terms of the (derived) tensor product of the homology of each space.
\item The stable-homotopical version of the previous example: the homology of a smash product of spectra is described in terms of the (derived) tensor product of the homology of each spectrum.
\end{itemize}
More generally, in any setting in which we have some notion of cohomology, and some kind of tensor product, one expects that a ``K\"{u}nneth formula'' in that setting ought to be some expression for the cohomology of a tensor product in terms of the (derived) tensor product of the cohomology of each factor.

Now suppose that $\Gamma$ is a bialgebra, or more generally, a bialgebroid, over some ground ring $A$. In the setting of $\Gamma$-comodules, the usual notion of cohomology is given by $\Cotor_{\Gamma}^*(A,-)$, the derived functors of the cotensor product $\Cotor_{\Gamma}^0(A,-) \cong A\Box_{\Gamma}-$. Using the multiplication on $\Gamma$, we get a tensor product $\otimes_A$ on the category of $\Gamma$-comodules. Consequently we would like to know if there exists a well-behaved K\"{u}nneth formula for the cotensor product, or more generally for $\Cotor$. There is apparently no place in the existing literature where such a K\"{u}nneth formula is considered. The purpose of this paper is to fill that ``hole'' in the literature.

Before explaining the main results of this paper, we pause to explain why it is both reasonable and unreasonable to expect some kind of K\"{u}nneth formula in $\Cotor$. Given an algebraic group $G$ over some field $k$, the representations of $G$ (taken in the broad sense, with no requirement of finite-dimensionality) are equivalent to the comodules over the representing Hopf algebra $kG$. The $\Cotor$-groups of the comodules recover the cohomology groups of the representations. Of course it is unrealistic to expect a simple relationship between $H^*(G; \rho_1),H^*(G; \rho_2),$ and $H^*(G; \rho_1\otimes_k \rho_2)$ for arbitrary representations $\rho_1$ and $\rho_2$! From this class of examples, it is clear that there cannot be a straightforward K\"{u}nneth formula in $\Cotor$ without some kind of restrictive hypotheses on the bialgebroid or on the comodules involved.

On the other hand, consider the occurrences of $\Cotor$ in algebraic topology. This is a paper in pure algebra, but motivated by topological questions, so we hope the reader will forgive a digression into topology. 
\begin{comment}
There are two standard computational tools in algebraic topology which require the use of $\Cotor$:
\begin{enumerate}\item
Given spaces $X,Y,Z$ and continuous functions $X\rightarrow Z$ and $Y\rightarrow Z$, in \cite{MR203730} Eilenberg and Moore construct the {\em Eilenberg-Moore spectral sequence}
\[ E^2_{*,*}  \cong \Cotor^{*,*}_{H_*(Z; k)}\left( H_*(X;k),H_*(Y;k)\right) \Rightarrow H_*(X\times^h_Z Y; k),\]
where $X\times^h_Z Y$ denotes the homotopy pullback of the maps $X\rightarrow Z$ and $Y\rightarrow Z$. 
\item 
\end{comment}
Given a generalized homology theory $E_*$ satisfying standard conditions\footnote{Standard references here are \cite{MR1324104} and appendix 1 of \cite{MR860042}.}, the ring $E_*E$ of stable cooperations on $E_*$-homology forms a bialgebroid $(E_*,E_*E)$. For any spectrum $X$ we get the {\em generalized Adams spectral sequence}
\begin{equation*}%\label{adams ss 2a} 
E_2^{*,*}  \cong \Cotor^{*,*}_{E_*E}\left( E_*,E_*(X)\right) \Rightarrow \pi_*\left( \hat{X}_E\right).\end{equation*}
In the particular case when $E$ is classical mod $p$ homology, we have $E_* \cong \mathbb{F}_p$ concentrated in degree zero, and we have $E_*E \cong A_*$, the mod $p$ dual Steenrod algebra, a commutative bialgebra over $\mathbb{F}_p$. When $X$ is a CW-complex with finitely many cells in each dimension, the spectral sequence then takes the form
\begin{equation*}%\label{adams ss 2a} 
E_2^{*,*}  \cong \Cotor^{*,*}_{A_*}\left( \mathbb{F}_p,H_*(X;\mathbb{F}_p)\right) \Rightarrow \left(\pi^{st}_*(X)\right)^{\hat{}}_p,\end{equation*}
where $\left(\pi^{st}_*(X)\right)^{\hat{}}_p$ is the $p$-adic completion of the stable homotopy groups of $X$, i.e., $\left(\pi^{st}_*(X)\right)^{\hat{}}_p \cong \lim_{n\rightarrow\infty} \pi^{st}_*(X)/p^n\pi^{st}_*(X)$. 
%\end{enumerate}
%More details about both spectral sequences are given in \cref{Topological applications.}.

Now suppose $X$ and $Y$ are spectra, and write $\pi_*^{st}(X\smash Y)^{\hat{}}_p$ for the $p$-adically completed stable homotopy groups of the smash product $X\smash Y$.
The image of the Hurewicz homomorphism 
\[ \pi_*^{st}(X\smash Y)^{\hat{}}_p\rightarrow H_*(X\smash Y; \mathbb{F}_p) \cong H_*(X;\mathbb{F}_p)\otimes_{\mathbb{F}_p}H_*( Y; \mathbb{F}_p)\]
is precisely the $0$-line in the Adams $E_{\infty}$-page, i.e., those elements in \[ \mathbb{F}_p\Box_{A_*}\left(H_*(X;\mathbb{F}_p)\otimes_{\mathbb{F}_p}H_*( Y; \mathbb{F}_p)\right) \]
which survive the Adams spectral sequence.

Mod $p$ homology {\em does} satisfy a K\"{u}nneth formula. The question of whether the {\em Hurewicz image in mod $p$ homology} satisfies a K\"{u}nneth formula is nearly\footnote{``Nearly'' here means ``up to Adams spectral sequence differentials originating on the $0$-line.''} the same question as asking whether $\Cotor^0$, i.e., the cotensor product, satisfies the K\"{u}nneth formula 
\[ \mathbb{F}_p\Box_{A_*}\left(H_*(X;\mathbb{F}_p)\otimes_{\mathbb{F}_p}H_*( Y; \mathbb{F}_p)\right) \cong \left( \mathbb{F}_p\Box_{A_*} H_*(X;\mathbb{F}_p)\right) \otimes_{\mathbb{F}_p}\left( \mathbb{F}_p\Box_{A_*}H_*( Y; \mathbb{F}_p)\right).
\]
An unbridled optimist might even hope for some kind of K\"{u}nneth theorem in higher $\Cotor$-groups as well, yielding a way to express the entire Adams $E_2$-page \[ \Cotor_{A_*}^{*,*}\left(\mathbb{F}_p, H_*(X;\mathbb{F}_p)\otimes_{\mathbb{F}_p} H_*(Y;\mathbb{F}_p)\right)\] for $X\smash Y$ in terms of the Adams $E_2$-pages for $X$ and for $Y$. 

The topological considerations make it seems plausible that a K\"{u}nneth formula could exist in $\Cotor^0$, or perhaps $\Cotor^*$, at least under appropriate hypotheses. On the other hand, the representation-theoretic considerations make it seem implausible. We hope the reader is now interested in knowing how the story turns out, which is the subject of this paper.

%Unfortunately, this last hope is too optimistic: we find in this paper that K\"{u}nneth theorems for $\Cotor$ require strong assumptions even for comparatively weak results, and it takes significant assumptions to even get an expression for $\mathbb{F}_p\Box_{A_*}H_*(X\smash Y;\mathbb{F}_p)$ in terms of $\mathbb{F}_p\Box_{A_*}H_*(X;\mathbb{F}_p)$ and $\mathbb{F}_p\Box_{A_*}H_*(Y;\mathbb{F}_p)$.

Our findings are outlined below:
\begin{itemize}
\item \Cref{Basic ideas...} reviews basic ideas about categories of comodules. The main result is a simple observation, Proposition \ref{A-linear cotensor}, which establishes that we cannot possibly have a K\"{u}nneth formula for comodules over a bialgebroid $\Gamma$ unless the left and right unit maps $\eta_L,\eta_R$ of $\Gamma$ are equal to one another, i.e., the bialgebroid $\Gamma$ is in fact a bialgebr{\em a}. Consequently we restrict our attention to comodules over bialgebras for the rest of the paper.
\item \Cref{The case when one comodule...} reviews the case where one comodule is trivial, i.e., all of its elements are primitive. This is the easiest case, and the results in \cref{The case when one comodule...} essentially all follow from a lemma that, as far as the author knows, first appeared in a 2002 paper of Al-Takhman, \cite{MR1910822}. The main result in \cref{The case when one comodule...} is Corollary \ref{al-takhman cor}, which establishes that, when $N$ is a {\em trivial} $\Gamma$-comodule which is flat over the ground ring $A$, we have isomorphisms 
\begin{equation}\label{kunneth thm 40949}\Cotor_{\Gamma}^n(L,M\otimes_A N) \cong \Cotor_{\Gamma}^n(L,M)\otimes_A N \cong \Cotor_{\Gamma}^n(L,M)\otimes_A (A\Box_{\Gamma} N)\end{equation}
for all $n$ and $L$.
\item Given that isomorphism \eqref{kunneth thm 40949} is a reasonable K\"{u}nneth-like theorem whenever $N$ is a trivial comodule, we can approach the situation when $M,N$ are each nontrivial by filtering $N$ so that its filtration quotients are trivial.
In \cref{The primitive filtration...} we carry out that idea. Specifically, in Definition-Proposition \ref{main defs}, we define a canonical such filtration, the ``primitive filtration,'' and if that filtration of a given comodule $N$ is exhaustive, then we get a canonical grading on $N$. In Proposition \ref{conn and exhaustive prim filt} we show that, when $\Gamma$ is a connected graded bialgebra, all bounded-below graded $\Gamma$-comodules have exhaustive primitive filtration; since the dual Steenrod algebra is connected and graded, the primitive filtration is exhaustive in the topological examples which motivate the investigations in this paper.

The result is a spectral sequence
\begin{align}
\label{kunneth ss 04000} E_1^{s,t} \cong \Cotor^s_{\Gamma}(L,M)\otimes_A N^t &\Rightarrow \Cotor^s_{\Gamma}(L,M\otimes_A N),
\end{align}
constructed in Theorem \ref{cotor sseq}. See that theorem for the necessary conditions for the existence of this spectral sequence, as well as the definition of the relevant grading on $N$ in the description of the $E_1$-page. Among several consequences of the existence of this spectral sequence, the most important is Corollary \ref{converse to al-takhman lemma}: if the canonical map
\begin{equation*}%\label{kunneth iso 2}
(L\Box_{\Gamma}M)\otimes_A (A\Box_{\Gamma}N) \rightarrow  L\Box_{\Gamma}(M\otimes_AN)\end{equation*} 
is an isomorphism, and if the canonical map 
\begin{equation*}%\label{kunneth map 1}
\Cotor_{\Gamma}^1(L,M)\otimes_A (A\Box_{\Gamma}N) \rightarrow  \Cotor_{\Gamma}^1(L,M\otimes_AN)\end{equation*} is also injective, then $N$ has trivial coaction. This is a negative result: it tells us that, unless one of the tensor factors is a comodule with trivial coaction, {\em we cannot expect any reasonable kind of K\"{u}nneth formula for both $\Cotor^0$ and $\Cotor^1$.} 
See Corollary \ref{converse to al-takhman lemma} for the mild hypotheses necessary for this result. 
\item In \cref{Example calculations...} we run spectral sequence \eqref{kunneth ss 04000} very explicitly in a family of examples where there are nontrivial differentials of several lengths. In particular, the spectral sequence calculations in \cref{Example calculations...} demonstrate that spectral sequence \eqref{kunneth ss 04000} does not necessarily collapse at $E_1$ or $E_2$, and is indeed capable of having arbitrarily long nonzero differentials.
\item \Cref{Kunneth formulas for...} takes up the question of when we have a K\"{u}nneth theorem for the cotensor product, i.e., $\Cotor^0$. In light of the negative result of Corollary \ref{converse to al-takhman lemma}, a K\"{u}nneth theorem for $\Cotor^0$ is as much as can be hoped for, unless one of the comodules in question is trivial. Theorem \ref{kunneth quot prop} is the main result in \cref{Kunneth formulas for...}: it establishes that, when $M$ is a subcomodule of $\Gamma$, we have a K\"{u}nneth isomorphism $(A\Box_{\Gamma}M)\otimes_A (A\Box_{\Gamma}N) \rightarrow A\Box_{\Gamma}(M\otimes_AN)$ if and only if, whenever $n\in N$ satisfies $\psi_N(n)\in M\otimes_A N\subseteq \Gamma\otimes_A N$, then $n$ is primitive. Here $\psi_N: N\rightarrow \Gamma\otimes_A N$ is the coaction map of the comodule $N$. This result is generalized by Theorem \ref{kunneth quot prop 2} and its corollaries, which allow $M$ to be a subcomodule of a finite (or, in the graded case, finite-type) direct sum of copies of $\Gamma$. 
\item Finally, in \cref{Topological applications.}, we give some topological consequences. In Corollaries \ref{main adams cor 1a} and \ref{main adams cor 2a}, we give necessary and sufficient conditions on a spectrum $Y$ for the $0$-line in the Adams $E_2$-page for $X\smash Y$ to decompose as a tensor product of the $0$-line in the Adams $E_2$-page for $X$ with the $0$-line in the Adams $E_2$-page for $Y$, with the criteria being especially explicit in the cases $X = BP, BP\langle n\rangle, ku, ko, $ and $tmf$.
\end{itemize}

%Theorem 4.11 of \cite{MR1885648} (which builds on and generalizes results in \cite{MR1707668}) gives us that, for a coalgebra $\Gamma$ over a field $k$ and for $\Gamma$-comodules $M$ and $N$, the Cotor-group $\Cotor^n_{\Gamma}(M,N)$ is isomorphic to the continuous Hochschild cohomology group $HH^n_c(A; M\otimes_k N)$, where $A$ is the profinite dual of $C$. So the results in this paper imply results about Kunneth theorems (for tensor products of coefficient bimodules, rather than tensor product of the algebras involved) for continuous Hochschild cohomology. However, the results we prove in this paper are a bit more general than that, since our results permit ground rings more general than fields.

%Similarly, given a field $k$, a $k$-coalgebra $\Gamma$, and $\Gamma$-comodules $M$ and $N$, the isomorphism $\left( \Cotor_{\Gamma}^n(M,N)\right)^* \cong \Tor^{\Gamma^*}_n(M^*,N^*)$ relates K\"{u}nneth theorems for Cotor to K\"{u}nneth theorems for $\Tor$. So some consequences of the results in this paper could be paraphrased as ``preduals'' of results about Kunneth formulas for $\Tor$. But again, what we do in this paper is a bit more general, since we allow ground rings more general than fields.

\section{Basic ideas, and the restriction to bialgebras.}
\label{Basic ideas...}

We recall a few basic facts about comodules; excellent references in the coalgebra case include the book-length treatment in \cite{MR2012570}, and in the bialgebroid case, Appendix 1 of \cite{MR860042}. Given a coalgebra or a bialgebroid $(A,\Gamma)$, a {\em left $\Gamma$-comodule} is a left $A$-module $M$ equipped with a coassociative, counital left $A$-module map $\psi: M \rightarrow\Gamma\otimes_A M$.
If $\Gamma$ is a bialgebroid\footnote{We discuss comodules over a bialgebroid, not the more general case of comodules over a coalgebroid, because one needs a multiplication on $\Gamma$ in order to define a tensor product of $\Gamma$-comodules over $A$. The argument is dual to the familiar argument that, if $R$ is a $k$-algebra and $M,N$ are $R$-modules, then we need to have a $k$-linear comultiplication on $R$ in order to get a natural $R$-module structure on $M\otimes_k N$. To be clear, the material on tensor product of comodules over a coalgebra in section 3.8 of \cite{MR2012570} is about the tensor product of a $\Gamma$-comodule with an $A$-module, rather than a tensor product of two $\Gamma$-comodules.} and flat over $A$, then we have an abelian category $\Comod(\Gamma)$ of left $\Gamma$-comodules equipped with a monoidal product given by tensor product over $A$, and the {\em relatively injective left $\Gamma$-comodules} are the retracts of those of the form $\Gamma \otimes_A M$ for some $A$-module $M$, which are called {\em extended comodules}. The extended comodule functor $E: \Mod(A)\rightarrow\Comod(\Gamma)$ is {\em right} adjoint to the forgetful functor $G: \Comod(\Gamma)\rightarrow \Mod(A)$. All derived functors in this paper are {\em relative} derived functors with respect to the allowable class whose injectives are the relative injectives. See chapter IX of \cite{MR1344215} for an excellent textbook treatment of relative homological algebra in general, or Appendix 1 of \cite{MR860042} (oriented toward generalized Adams spectral sequences) or section 10.11 of \cite{MR2604913} (oriented toward Eilenberg-Moore spectral sequences) for an introductory treatment in the present setting, that is, relative derived functors with respect to this particular allowable class on comodules\footnote{The reader who prefers not to deal with relative homological algebra may be relieved to know that, when the ground ring $A$ of the bialgebroid $(A,\Gamma)$ is a field, these relative derived functors agree with ordinary derived functors. So relative homological algebra need not be mentioned at all unless the ground ring $A$ fails to be a field.}.

Here is the basic question: under what conditions do we have a K\"{u}nneth formula for comodule primitives? A first attempt at making such a question precise is as follows: given left $\Gamma$-comodules $M$ and $N$, we ask under what conditions we might have an isomorphism
\begin{equation}\label{kunneth iso 1} A\Box_{\Gamma}\left( M\otimes_A N\right) \cong \left( A\Box_{\Gamma} M\right)\otimes_A \left( A\Box_{\Gamma} N\right).\end{equation}

The most obvious requirement is that $\left( A\Box_{\Gamma} M\right)\otimes_A \left( A\Box_{\Gamma} N\right)$ must actually be {\em defined}. So $A\Box_{\Gamma} M$ and $A\Box_{\Gamma} N$ must be $A$-modules. Here is the relevant observation:
\begin{prop}\label{A-linear cotensor}
Let $(A,\Gamma)$ be a bialgebroid. Then the following conditions are equivalent:
\begin{enumerate}
\item The subgroup inclusion \begin{equation}\label{incl 238} M\Box_{\Gamma}N \hookrightarrow M\otimes_A N\end{equation} is $A$-linear for every right $\Gamma$-comodule $M$ and left $\Gamma$-comodule $N$.
\item The subgroup inclusion
\begin{equation}\label{incl 239} A\Box_{\Gamma}M \hookrightarrow M\end{equation} is $A$-linear for all left $\Gamma$-comodules $M$.
\item The bialgebroid $(A,\Gamma)$ is a bialgebra. That is, the left unit and right unit maps $\eta_L,\eta_R$ of $(A,\Gamma)$ coincide, i.e., $\eta_L = \eta_R$.
\end{enumerate}
\end{prop}
\begin{proof}\leavevmode
\begin{description}
\item[1 implies 2] Trivial.
\item[2 implies 3] Let $M = A$. Then $1\in A\Box_{\Gamma}A$, so \eqref{incl 239} implies in particular that 
\begin{align*}
 \eta_R(a) \otimes 1 
  &= 1 \otimes a \\
  &= \psi(a\cdot 1) \\
  &= a\psi(1) \\
  &= a(1\otimes 1) \\
  &= \eta_L(a)\otimes 1 \in \Gamma\otimes_A A,
\end{align*}
so $\eta_R(a) = \eta_L(a)$ for all $a\in A$.
Here $\psi$ is the coaction map on $M$.
\item[3 implies 1] Recall that the structure map $\psi: M \rightarrow\Gamma\otimes_A M$ of a left $\Gamma$-comodule $M$ is required to be a {\em left} $A$-module morphism. If $\eta_L = \eta_R$, then $\psi$ is a left $A$-module morphism if and only if it is a right $A$-module morphism. Consequently the maps
$\psi_M\otimes N$ and $M\otimes \psi_N$ in the equalizer sequence
\begin{equation*}%\label{}
\xymatrix{
 M\Box_{\Gamma} N 
  \ar[r] 
   & 
  M\otimes_A N  
   \ar@<1ex>[r]^{\psi_M\otimes N}\ar@<-1ex>[r]_{M\otimes \psi_N} &
  M\otimes_A \Gamma\otimes_A N  
}\end{equation*}
are $A$-module morphisms, and so the inclusion \eqref{incl 238} is $A$-linear.
\end{description}
\end{proof}
\begin{remark}\label{kunneth counterexample}
As a consequence of Proposition \ref{A-linear cotensor}, we could have a K\"{u}nneth isomorphism \eqref{kunneth iso 1} for all $\Gamma$-comodules $M,N$ only if $(A,\Gamma)$ is a bialgebr{\em a}, not only a bialgebr{\em oid}. But even restricting to comodules over a bialgebra, the formula \eqref{kunneth iso 1} still often fails to hold. For example, if we let $A=k$ for some field $k$, let $\Gamma = k[x]$ with $x$ primitive, and let $M = N$ be the sub-$\Gamma$-comodule of $\Gamma$ which is $k$-linearly spanned by $1$ and $x$, then $A\Box_{\Gamma}M \cong A\Box_{\Gamma}N \cong k$ spanned by $1$, while $A\Box_{\Gamma}(M\otimes_A N) \cong k\oplus k$ with basis $\{ 1\otimes 1, x\otimes 1 - 1\otimes x\}$. Since the $k$-vector-space dimensions differ, there is no way we could have an isomorphism exactly of the form \eqref{kunneth iso 1}. But the ``moral'' of Proposition \ref{A-linear cotensor} remains true: any K\"{u}nneth formula for $\Cotor$---i.e., any way of describing $\Cotor$ of a tensor product over $A$ in terms of a tensor product over $A$ of $\Cotor$-modules---first requires that $A$ be a bialgebra, so that the tensor product over $A$ of $\Cotor$-modules is defined at all. 

%The rest of this paper is, in large part, devoted to determining when isomorphism \eqref{kunneth iso 1} {\em does} hold.
\end{remark}

%\begin{remark} While most of the literature on bialgebras $(A,\Gamma)$ assumes that $A$ is a field, the tradition of allowing $A$ to merely be a commutative ring, as we do in Proposition \ref{A-linear cotensor}, is an old one, appearing already in \cite{MR0174052}. \end{remark}

In light of Proposition \ref{A-linear cotensor}, %the question of K\"{u}nneth theorems for comodules is not more general over bialgebroids than over bialgebras, so 
we restrict our attention to comodules over bialgebras for the rest of this paper.

\section{The case where one comodule is trivial.}
\label{The case when one comodule...}
We have a tensor-cotensor relation given by the result\footnote{Proposition \ref{al-takhman result} is very easy to prove---indeed, it appears as a lemma with a five-line proof in the earliest published paper where I am aware that it appears, \cite{MR1910822}---but since Proposition \ref{al-takhman result} is the closest thing to a K\"{u}nneth theorem for $\Cotor$ which is already in the literature, we hope the reader will forgive us to trying to give a bit of history of this lemma. After appearing as Lemma 3.8 in \cite{MR1910822}, it shows up as Lemma 2.3 in \cite{MR1916479}, the published form of Al-Takhman's D\"{u}sseldorf thesis \cite{altakhmanthesis}, which is easier to locate than \cite{MR1910822}, and finally it appears as 10.6 in the book \cite{MR2012570}.}
\begin{prop}\label{al-takhman result}
If $\Gamma$ is a coalgebra\footnote{The isomorphisms \eqref{map 1} and \eqref{map 2} appear to involve tensor products, over $A$, of $\Gamma$-comodules, and yet $\Gamma$ is only a coalgebra, not a bialgebra. The reason these tensor products make sense, as $\Gamma$-comodules, despite the lack of a multiplication on $\Gamma$ is that in each case, one of the tensor factors is a {\em trivial} comodule.} over a commutative ring $A$, $M$ is a right $\Gamma$-comodule, and $N$ is a left $\Gamma$-comodule, and $W$ is an $A$-module, then we have natural $A$-linear maps
\begin{align}
\label{map 1} W\otimes_A (M\Box_{\Gamma}N) &\rightarrow (W\otimes_A M)\Box_{\Gamma} N, \\
\label{map 2} (M\Box_{\Gamma}N)\otimes_A W &\rightarrow M\Box_{\Gamma} (N\otimes_A W),
\end{align}
where $W\otimes_A M$ and $N\otimes_A W$ are each given the {\em trivial} $\Gamma$-comodule structure, i.e., the coaction maps are 
\begin{align*}
 W\otimes_A M &\rightarrow W\otimes_A M\otimes_A \Gamma \\
 w\otimes m &\rightarrow w \otimes \psi_M(m), \mbox{\ \ \ and} \\
 N\otimes_A W &\rightarrow \Gamma \otimes_A N\otimes_A W\\
 n\otimes w &\rightarrow \psi_N(n) \otimes w,
\end{align*}
where $\psi_M: M \rightarrow M\otimes_A \Gamma$ and $\psi_N: N \rightarrow \Gamma\otimes_A N$ are the coaction maps of $M$ and $N$, respectively.

Furthermore, the following conditions are equivalent:
\begin{itemize}
\item The map \eqref{map 1} is an isomorphism.
\item The map \eqref{map 2} is an isomorphism.
\item The canonical inclusion $M\Box_{\Gamma}N \hookrightarrow M\otimes_AN$ remains injective after applying the functor $-\otimes_A W$.
\end{itemize}
\end{prop}

While Proposition \ref{al-takhman result} is the closest result to a K\"{u}nneth theorem for $\Cotor$ (in this case, only $\Cotor^0$, i.e., the cotensor product) appearing in the literature, it has a critical limitation that we need to overcome: it describes $\Cotor_{\Gamma}^*(L,M\otimes_A N)$ in terms of $\Cotor_{\Gamma}^*(L,M)$ and $N$ only when either $M$ or $N$ has trivial $\Gamma$-coaction. Here a left $\Gamma$-comodule $W$ is said to have {\em trivial coaction} if its coaction map $\psi_W: W \rightarrow\Gamma \otimes_A W$ is given by $\psi_W(w) = 1\otimes w$. 
This limitation is far too strict for Proposition \ref{al-takhman result} to be useful for the topological applications of $\Cotor$. One would like to find a description of $\Cotor_{\Gamma}^*(L,M\otimes_A N)$ in terms as close as possible to $\Cotor_{\Gamma}^*(L,M)$ and $\Cotor_{\Gamma}^*(L,N)$, {\em without} assuming that either $M$ or $N$ have trivial $\Gamma$-coaction.

Before moving on to the case where $M$ and $N$ each have nontrivial coaction, we at least remark that Proposition \ref{al-takhman result} has an easy corollary for the higher $\Cotor$ groups.
We first need a couple of easy lemmas. Recall that $E: \Mod(A)\rightarrow\Comod(\Gamma)$ denotes the extended comodule functor, which is right adjoint to the forgetful functor $G: \Comod(\Gamma)\rightarrow\Mod(A)$. Then:
\begin{lemma}\label{doi lemma}
There exists an isomorphism $E(M\otimes_A GN) \cong (EM)\otimes_A N$ of $\Gamma$-comodules, natural in the variables $M$ and $N$.
\end{lemma}
\begin{proof}
See Proposition 9 in \cite{MR597479}, where this is proven for $A$ a bialgebra; the same proof applies when $A$ is a bialgebroid.
\end{proof}

\begin{lemma}\label{rel injs are an ideal}
If $M,N$ are left $\Gamma$-comodules and $M$ is relatively injective, $M\otimes_A N$ is also relatively injective.
\end{lemma}
\begin{proof}
Choose an $A$-module $\tilde{M}$ and left $\Gamma$-comodule morphisms $i: M \rightarrow E\tilde{M}$ and $\pi: E\tilde{M} \rightarrow M$ such that $\pi\circ i = \id_M$. Then we have the commutative diagram
\[\xymatrix{
 M\otimes_A N \ar@/_1pc/[rr]_{\id} \ar[r]^{i\otimes_A N} & (E \tilde{M})\otimes_A N \ar[r]^{\pi\otimes_A N} & M\otimes_A N ,
}\]
and $(E \tilde{M})\otimes_A N\cong E(\tilde{M}\otimes_A N)$ as left $\Gamma$-comodules, by Lemma \ref{doi lemma}. So $M\otimes_A N$ is indeed a retract of the extended comodule $E(\tilde{M}\otimes_A N)$.
\end{proof}

Now we have the following corollary of Proposition \ref{al-takhman result}:
\begin{corollary}\label{al-takhman cor}
If $\Gamma$ is a coalgebra over a commutative ring $A$, $L$ is a right $\Gamma$-comodule, $M$ is a left $\Gamma$-comodule, and $N$ is a {\em trivial} left $\Gamma$-comodule which is flat over $A$, then we have isomorphisms of $A$-modules
\begin{align*}
 \Cotor_{\Gamma}^n(L,M\otimes_A N) 
 &\cong \Cotor_{\Gamma}^n(L,M)\otimes_A N\\
 &\cong \Cotor_{\Gamma}^n(L,M)\otimes_A (A\Box_{\Gamma} N)\end{align*}
for each nonnegative integer $n$.

In particular, in the $n=0$ case we have a K\"{u}nneth isomorphism
\[ L\Box_{\Gamma}(M\otimes_A N) \cong (L\Box_{\Gamma} M) \otimes_A (A\Box_{\Gamma} N).\] 
\end{corollary}
\begin{proof}
Since $N$ has trivial coaction, we have an isomorphism of left $\Gamma$-comodules $N\cong A\Box_{\Gamma}N$, which we use freely throughout this proof.

Given a resolution $I^{\bullet}$ of $M$ by relatively injective left $\Gamma$-comodules, the cochain complex of left $A$-comodules $I^{\bullet}\otimes_A N$ is exact by flatness of $N$, and it is a complex of relative injectives due to Lemma \ref{rel injs are an ideal}. Consequently the cohomology of the cochain complex $L\Box_{\Gamma}(I^{\bullet}\otimes_A N)$ yields the $\Cotor$-groups $\Cotor^*_{\Gamma}(L,M\otimes_A N)$. Applying Proposition \ref{al-takhman result} yields the isomorphism of cochain complexes $L\Box_{\Gamma}(I^{\bullet}\otimes_A N) \cong (L\Box_{\Gamma}I^{\bullet})\otimes_A N$, and the cohomology of the cochain complex on the right is $\Cotor_{\Gamma}^*(L,M)\otimes_AN$, again using flatness of $N$.
\end{proof}
Corollary \ref{al-takhman cor} has a partial converse given below by Corollaries \ref{converse to al-takhman lemma} and \ref{graded converse to al-takhman lemma}.

\section{The primitive filtration of a comodule.}
\label{The primitive filtration...}

A basic idea in what follows is that, when the comodules $M,N$ both have nontrivial coaction, the results of \cref{The case when one comodule...} do not directly apply, but we could try to filter one of the comodules $M$ or $N$ so that each of the filtration quotients has trivial coaction, in order to get a spectral sequence whose input term could be simplified by some application of Corollary \ref{al-takhman cor}.
There is a canonical and quite useful such filtration\footnote{This primitive filtration on a comodule is {\em not} the same as the primitive filtration on a bialgebra, from \cite{MR0174052}. We have never seen our primitive filtration in the literature and have never heard it mentioned by others, but it is a very simple and effective idea, and we expect it has probably been considered by others on more than occasion.}, defined in Definition-Proposition \ref{main defs}:
\begin{definition-proposition}\leavevmode    \label{main defs}
\begin{itemize}
\item
Let $(A,\Gamma)$ be a bialgebroid, and let $M$ be a left $\Gamma$-comodule. By the {\em primitive filtration of $M$} we mean the filtration 
\begin{equation}\label{primitive filt} 0 = M_{-1} \subseteq M_0 \subseteq M_1 \subseteq \dots \subseteq M\end{equation}
of $M$ by subgroups defined as follows: $M_i$ is the kernel of the projection $M \rightarrow M(i+1)$, where 
\begin{equation}\label{surj seq} M = M(0) \rightarrow M(1) \rightarrow M(2) \rightarrow \dots \end{equation}
is a sequence of surjective group homomorphisms defined inductively by letting $M(i+1)$ be the cokernel of the inclusion $A\Box_{\Gamma} M(i) \hookrightarrow M(i)$.
Consequently we have short exact sequences of groups
\begin{eqnarray}
\label{ses 004040}
 0 \rightarrow M_i \rightarrow M \rightarrow M(i+1) \rightarrow 0,\\
\label{ses 004041}
 0 \rightarrow A\Box_{\Gamma}M(i) \rightarrow M(i) \rightarrow M(i+1) \rightarrow 0,\mbox{\ \ \ and}\\
\label{ses 004042}
 0 \rightarrow M_{i} \rightarrow M_{i+1} \rightarrow A\Box_{\Gamma}M(i+1) \rightarrow 0.
\end{eqnarray}
\item
We will say that {\em $M$ has exhaustive primitive filtration} if $\cup_i M_i = M$. We will say that {\em $M$ has finite primitive filtration} if $M_i = M$ for some $M$.
\item
If $(A,\Gamma)$ is a bialgebra, then the primitive filtration \eqref{primitive filt} is a filtration by sub-$\Gamma$-comodules, not merely by subgroups, and the extensions \eqref{ses 004040},\eqref{ses 004041}, and \eqref{ses 004042} are extensions of $\Gamma$-comodules.
\item
When the underlying $A$-module extensions of the left $\Gamma$-comodule extensions \eqref{ses 004042} are split, we say that primitive filtration on $M$ is {\em split.} 
The primitive filtration on $M$ is automatically split if, for example 
%$(A,\Gamma)$ is a bialgebra and 
the commutative ring $A$ is semisimple, e.g. a field.
\item
When the primitive filtration on $M$ is split and exhaustive, then we have an isomorphism of $A$-modules\footnote{But generally not an isomorphism of $\Gamma$-comodules.} $M\cong \coprod_{i\geq 0} A\Box_{\Gamma}M(i)$. We then refer to this grading on $M$, whose degree $i$ summand $M^i$ is $A\Box_{\Gamma}M(i)$, as the {\em primitive grading} on $M$.
\end{itemize}
\end{definition-proposition}
\begin{proof}
By Proposition \ref{A-linear cotensor}, if $(A,\Gamma)$ is a bialgebra, then each of the inclusions $A\Box_{\Gamma} M(i) \hookrightarrow M(i)$ is a left $\Gamma$-comodule morphism, and consequently \eqref{surj seq} is a sequence of left $\Gamma$-comodule morphisms, so \eqref{primitive filt} is also a sequence of left $\Gamma$-comodule morphisms, and similarly for \eqref{ses 004040},\eqref{ses 004041}, and \eqref{ses 004042}.
\end{proof}

Of course every left $\Gamma$-comodule whose underlying $A$-module is Artinian has finite primitive filtration. Another useful source of comodules with well-behaved primitive filtrations is Proposition \ref{conn and exhaustive prim filt}:
\begin{prop}\label{conn and exhaustive prim filt}
Let $(A,\Gamma)$ be a graded bialgebra which is connected, i.e., $A$ and $\Gamma$ each are trivial in negative degrees and the unit map $\eta: A \rightarrow \Gamma$ is surjective (equivalently, an isomorphism) in grading degree zero. Then every bounded-below\footnote{It is standard that a graded group, module, ring, etc. $M$ is said to be {\em bounded below} if there exists an integer $n$ such that $M$ is trivial in grading degrees $<n$.} graded left $\Gamma$-comodule has exhaustive primitive filtration.
\end{prop}
\begin{proof}
If $M$ is a bounded-below graded left $\Gamma$-comodule which is trivial in grading degrees $<n$, then $A\Box_{\Gamma} M \rightarrow M$ is an isomorphism in the bottommost grading degree, so its cokernel is also bounded-below and with a strictly higher lower bound on the grading degrees of its nontrivial summands. By induction, in \eqref{surj seq} each $M(i)$ is trivial in grading degrees $<n+i$, and in \eqref{primitive filt}, the inclusion $M_i\subseteq M$ is an isomorphism in grading degrees $<n+i-1$. So every homogeneous element of $M$ is contained in $M(i)$ for some $i$.
\end{proof}

We need a couple of easy lemmas, Lemma \ref{easy lemma 1} and \ref{easy lemma 2}, which are certainly not new:
\begin{lemma}\label{easy lemma 1}
Let $(A,\Gamma)$ be a bialgebra, let $I$ be a set, and let $\{ M_i: i\in I\}$ be a set of left $\Gamma$-comodules. Let $L$ be a left $\Gamma$-comodule. Then the natural map\footnote{Since the forgetful functor from $\Gamma$-comodules to $A$-modules is a left adjoint, it preserves colimits, and so the coproduct in the domain of \eqref{nat map 3049498} can equally well be regarded as a coproduct in $A$-modules or a coproduct in $\Gamma$-comodules. Of course the coproduct in the codomain of \eqref{nat map 3049498} must be regarded as a coproduct in $\Gamma$-comodules, since otherwise it would not make sense to apply the cotensor product to it.}
\begin{equation}\label{nat map 3049498} \coprod_{i\in I}\left( L\Box_{\Gamma}M_i\right) \rightarrow L\Box_{\Gamma}\left( \coprod_{i\in I} M_i\right) \end{equation}
is an isomorphism.
\end{lemma}
\begin{proof}
We have a commutative diagram of $A$-modules with exact rows
\begin{equation}\label{diag 3049999}
\xymatrix{
 0 \ar[r]\ar[d] & \coprod_i L\Box_{\Gamma} M_i \ar[r]\ar[d] & \coprod_i L\otimes_A M_i \ar[r]\ar[d]^{\cong} & \coprod_i L\otimes_A \Gamma\otimes_A M_i \ar[d]^{\cong} \\
 0 \ar[r] & L\Box_{\Gamma} \coprod_i M_i \ar[r] & L\otimes_A \coprod_i M_i \ar[r] & L\otimes_A \Gamma\otimes_A \coprod_i M_i
}\end{equation}
and the maps indicated as isomorphisms are isomorphisms since tensor products, in $A$-modules, commute with coproducts. Hence the remaining vertical map in \eqref{diag 3049999}---i.e., the map \eqref{nat map 3049498}---is also an isomorphism.
\end{proof}

\begin{lemma}\label{easy lemma 2}
Let $(A,\Gamma)$ be a bialgebra, and let
\begin{equation}\label{nat map 3049397} M_0\rightarrow M_1 \rightarrow M_2 \rightarrow\dots\end{equation} be a sequence of morphisms of left $\Gamma$-comodules. Let $L$ be a left $\Gamma$-comodule, and let $n$ be a nonnegative integer. Then the natural group homomorphism
\begin{equation}\label{nat map 3049499} \colim_i\Cotor^n_{\Gamma}(L,M_i)\rightarrow \Cotor^n_{\Gamma}(L,\colim_i M_i) \end{equation}
is an isomorphism.
\end{lemma}
\begin{proof}
Let $D_{\Gamma}^{\bullet}(M_i)$ denote the unreduced cobar resolution of $M_i$ (as in Definition A1.2.11 of Appendix 1 of \cite{MR860042}, but using $\Gamma$ in place of the unit coideal $\ker \epsilon$). Since $D_{\Gamma}^n(M_i)\cong \Gamma^{\otimes_A (n+1)}\otimes_A M_i$ for each $n$, the natural map of cochain complexes of $A$-modules \begin{equation}\label{nat map 3049500}\colim_i D_{\Gamma}^{\bullet}(M_i)\rightarrow D_{\Gamma}^{\bullet}(\colim_i M_i)\end{equation} is an isomorphism. Since the forgetful functor from $\Gamma$-comodules to $A$-modules is faithful and a left adjoint, it preserves and reflects colimits, so \eqref{nat map 3049500} is in fact an isomorphism of cochain complexes of $\Gamma$-comodules. 

In the case that the colimit is merely a coproduct (i.e., the maps \eqref{nat map 3049397} are split injections), applying Lemma \ref{easy lemma 1} to \eqref{nat map 3049500} yields that 
\begin{equation}\label{nat map 3049495}\coprod_i\Cotor^n_{\Gamma}(L,M_i)\rightarrow \Cotor^n_{\Gamma}(L,\coprod_i M_i)\end{equation} is an isomorphism.

More generally---i.e., when the maps in \eqref{nat map 3049397} are not necessarily split injections---we have the short exact sequence of $\Gamma$-comodules
\begin{equation}\label{nat map 3049494}\coprod_i M_i \stackrel{\id - T}{\longrightarrow} \coprod_i M_i \rightarrow \colim_i M_i \end{equation}
and, applying $\Cotor_{\Gamma}^*(L,-)$ to \eqref{nat map 3049494}, a commutative diagram with exact columns
\[\xymatrix{
 \vdots \ar[d] & 
  \vdots \ar[d] \\
 \coprod_i\Cotor^{n}_{\Gamma}(L, M_i) \ar[r]^{\cong} \ar[d] &
  \Cotor^{n}_{\Gamma}(L, \coprod_iM_i) \ar[d] \\
 \coprod_i\Cotor^{n}_{\Gamma}(L, M_i) \ar[r]^{\cong} \ar[d] &
  \Cotor^{n}_{\Gamma}(L, \coprod_iM_i) \ar[d] \\
 \colim_i\Cotor^{n}_{\Gamma}(L, M_i) \ar[r] \ar[d] &
  \Cotor^{n}_{\Gamma}(L, \colim_iM_i) \ar[d] \\
 \coprod_i\Cotor^{n+1}_{\Gamma}(L, M_i) \ar[r]^{\cong} \ar[d] &
  \Cotor^{n+1}_{\Gamma}(L, \coprod_iM_i) \ar[d] \\
 \coprod_i\Cotor^{n+1}_{\Gamma}(L, M_i) \ar[r]^{\cong} \ar[d] &
  \Cotor^{n+1}_{\Gamma}(L, \coprod_iM_i) \ar[d] \\
 \vdots  & 
  \vdots 
}\]
in which the indicated maps are isomorphisms due to \eqref{nat map 3049495} being an isomorphism. Now the Five Lemma gives us that the remaining map---i.e., \eqref{nat map 3049499}---is an isomorphism.
\end{proof}

\begin{prop}\label{tor sseq}
Let $(A,\Gamma)$ be a bialgebra, and let $M$ be a left $\Gamma$-comodule with exhaustive primitive filtration $M_0 \subseteq M_1 \subseteq \dots$. Then we have a first-quadrant spectral sequence
\begin{align*}
 E^1_{s,t} \cong \Tor_s^A(A\Box_{\Gamma} M(t),N) &\Rightarrow \Tor_s^A(M,N) \\
 d^r: E^r_{s,t} &\rightarrow E^r_{s-1,t+r}.
\end{align*}
\end{prop}
\begin{proof}
This is simply the spectral sequence of the exact couple obtained by applying $\Tor_*^A(-,N)$ to the tower of extensions of $A$-modules
\[\xymatrix{
 M_{-1}\ar[r] & M_0 \ar[r]\ar[d] & M_1 \ar[r]\ar[d] & M_2 \ar[r]\ar[d] & \dots \\
 & A\Box_{\Gamma}M(0) & A\Box_{\Gamma}M(1)& A\Box_{\Gamma}M(2) . &
}\]
\end{proof}

Of course, in most motivating examples of bialgebras $(A,\Gamma)$, $A$ is a field and consequently the spectral sequence of Proposition \ref{tor sseq} collapses to the $s=0$ line already at the $E^1$-page, yielding an unsurprising isomorphism. So the spectral sequence of Proposition \ref{tor sseq} is not our focus at all, and we mention it only for completeness. A much more interesting spectral sequence is given by the following:
\begin{theorem}\label{cotor sseq}{\bf (The K\"{u}nneth spectral sequence for Cotor.)} 
Let $(A,\Gamma)$ be a bialgebra, and let $L$ be a right $\Gamma$-comodule and $M,N$ left $\Gamma$-comodules. Suppose that $M$ is flat over $A$, and suppose that $N$ has exhaustive primitive filtration.

Then we have a %conditionally convergent 
first quadrant spectral sequence
\begin{align}
\label{kunneth ss 1} E_1^{s,t} \cong \Cotor^s_{\Gamma}(L,M)\otimes_A \left( A\Box_{\Gamma} N(t)\right) &\Rightarrow \Cotor^s_{\Gamma}(L,M\otimes_A N)\\
\nonumber d_r: E_r^{s,t} &\rightarrow E_r^{s+1,t-r}.
\end{align}
If the primitive filtration on $N$ is also split (for example, if $A$ is a field), then the $E_1$-term of \eqref{kunneth ss 1} is also given by
$E_1^{s,t} \cong \Cotor^s_{\Gamma}(L,M)\otimes_A N^t$
where $N^t = A\Box_{\Gamma}N(t)$ is the degree $t$ summand in the primitive grading on $N$, as in Definition-Proposition \ref{main defs}. %The spectral sequence is strongly convergent if the primitive filtration on $N$ is finite.
\end{theorem}
\begin{proof}
Since $M$ is flat, applying $M\otimes_A -$ to the primitive filtration on $N$ yields a tower of extensions of left $\Gamma$-comodules
\[\xymatrix{
 0 = M\otimes_A N_{-1} \ar[r] & M\otimes_A N_0 \ar[r]\ar[d] & M\otimes_A N_1 \ar[r]\ar[d] & \dots \\
 & M\otimes_A \left( A\Box_{\Gamma}N(0)\right) &  M\otimes_A \left( A\Box_{\Gamma}N(1)\right) & 
}\]
such that $\colim_{n\rightarrow\infty} M\otimes_A N_n\rightarrow M\otimes_A N$ is an isomorphism,
and consequently a spectral sequence 
\begin{align}
\label{sseq 03940439} E_1^{s,t} \cong \Cotor^s_{\Gamma}\left(L,M\otimes_A \left( A\Box_{\Gamma}N(t)\right)\right) &\Rightarrow \colim_i \Cotor^s_{\Gamma}(L,M\otimes_A N_i)\\
\nonumber d_r: E_r^{s,t} &\rightarrow E_r^{s+1,t-r}.
\end{align}
Recall that $A\Box_{\Gamma} N(t)$ has a natural left $\Gamma$-coaction given by Proposition \ref{A-linear cotensor}, since $(A,\Gamma)$ is a bialgebra, but that left $\Gamma$-coaction is clearly the trivial one, since every element of $A\Box_{\Gamma} N(t)$ is primitive.
Triviality of the coaction is what allows us to apply Corollary \ref{al-takhman cor}, to get that $\Cotor^s_{\Gamma}\left(L,M\otimes_A \left(A\Box_{\Gamma}N(t)\right)\right)\cong \Cotor^s_{\Gamma}(L,M)\otimes_A \left( A\Box_{\Gamma}N(t)\right)$ for each $s,t$.
Finally, the abutment $\colim_i \Cotor^*_{\Gamma}(L,M\otimes_A N_i)$ of
spectral sequence \eqref{sseq 03940439} is isomorphic to 
\[ \Cotor^*_{\Gamma}\left(L,\colim_i (M\otimes_A N_i)\right)) \cong \Cotor^*_{\Gamma}(L,M\otimes_A N)\] 
by Lemma \ref{easy lemma 2}.
\end{proof}

In case it helps the reader to visualize the spectral sequence, we provide a picture of a portion of the $E_1$-page of spectral sequence of Theorem \ref{cotor sseq}, with $s$ as the horizontal coordinate, $t$ as the vertical coordinate, the $d_1$-differentials colored in red, the $d_2$-differentials colored in orange, and the $d_3$-differentials colored in blue:
\noindent
\begin{equation}\label{ss picture 1}\begin{tikzpicture}[trim left=0cm,xscale=3.8,yscale=0.7]
\draw[->,color=red] (0,1) -- ($(0,1)!0.9!(1,0)$) node{};
\draw[->,color=red] (0,2) -- ($(0,2)!0.9!(1,1)$) node{};
\draw[->,color=red] (0,3) -- ($(0,3)!0.9!(1,2)$) node{};
\draw[->,color=red] (1,1) -- ($(1,1)!0.9!(2,0)$) node{};
\draw[->,color=red] (1,2) -- ($(1,2)!0.9!(2,1)$) node{};
\draw[->,color=red] (1,3) -- ($(1,3)!0.9!(2,2)$) node{};
\draw[->,color=red] (2,1) -- ($(2,1)!0.9!(3,0)$) node{};
\draw[->,color=red] (2,2) -- ($(2,2)!0.9!(3,1)$) node{};
\draw[->,color=red] (2,3) -- ($(2,3)!0.9!(3,2)$) node{};
\draw[->,color=orange] (0,2) -- ($(0,2)!0.9!(1,0)$) node{};
\draw[->,color=orange] (0,3) -- ($(0,3)!0.9!(1,1)$) node{};
\draw[->,color=orange] (1,2) -- ($(1,2)!0.9!(2,0)$) node{};
\draw[->,color=orange] (1,3) -- ($(1,3)!0.9!(2,1)$) node{};
\draw[->,color=orange] (2,2) -- ($(2,2)!0.9!(3,0)$) node{};
\draw[->,color=orange] (2,3) -- ($(2,3)!0.9!(3,1)$) node{};
\draw[->,color=blue] (0,3) -- ($(0,3)!0.9!(1,0)$) node{};
\draw[->,color=blue] (1,3) -- ($(1,3)!0.9!(2,0)$) node{};
\draw[->,color=blue] (2,3) -- ($(2,3)!0.9!(3,0)$) node{};
\draw (-0.35,3.5) -- (-0.35,-0.35) -- (3.5,-0.35);
\draw (-0.5,0) node{$t=0$};
\draw (-0.5,1) node{$t=1$};
\draw (-0.5,2) node{$t=2$};
\draw (-0.5,3) node{$t=3$};
\draw (0,-0.6) node{$s=0$};
\draw (1,-0.6) node{$s=1$};
\draw (2,-0.6) node{$s=2$};
\draw (3,-0.6) node{$s=3$};
\draw (0,0) node{$(L\Box_{\Gamma}M)\otimes_A N^0$};
\draw (1,0) node{$\Cotor^1_{\Gamma}(L,M)\otimes_A N^0$};
\draw (2,0) node{$\Cotor^2_{\Gamma}(L,M)\otimes_A N^0$};
\draw (3,0) node{$\Cotor^3_{\Gamma}(L,M)\otimes_A N^0$};
\draw (3.6,0) node{$\dots$};
\draw (0,1) node{$(L\Box_{\Gamma}M)\otimes_A N^1$};
\draw (1,1) node{$\Cotor^1_{\Gamma}(L,M)\otimes_A N^1$};
\draw (2,1) node{$\Cotor^2_{\Gamma}(L,M)\otimes_A N^1$};
\draw (3,1) node{$\Cotor^3_{\Gamma}(L,M)\otimes_A N^1$};
\draw (3.6,1) node{$\dots$};
\draw (0,2) node{$(L\Box_{\Gamma}M)\otimes_A N^2$};
\draw (1,2) node{$\Cotor^1_{\Gamma}(L,M)\otimes_A N^2$};
\draw (2,2) node{$\Cotor^2_{\Gamma}(L,M)\otimes_A N^2$};
\draw (3,2) node{$\Cotor^3_{\Gamma}(L,M)\otimes_A N^2$};
\draw (3.6,2) node{$\dots$};
\draw (0,3) node{$(L\Box_{\Gamma}M)\otimes_A N^3$};
\draw (1,3) node{$\Cotor^1_{\Gamma}(L,M)\otimes_A N^3$};
\draw (2,3) node{$\Cotor^2_{\Gamma}(L,M)\otimes_A N^3$};
\draw (3,3) node{$\Cotor^3_{\Gamma}(L,M)\otimes_A N^3$};
\draw (3.6,3) node{$\dots$};
\draw (0,4) node{$\vdots$};
\draw (1,4) node{$\vdots$};
\draw (2,4) node{$\vdots$};
\draw (3,4) node{$\vdots$};
\end{tikzpicture} \end{equation}
In \eqref{ss picture 1} we write $N^i$ rather than $A\Box_{\Gamma}N(i)$ as though the primitive filtration on $N$ is split, but this is only for notational convenience, to make the illustration of the spectral sequence more readable. If the primitive filtration on $N$ is not split, then replace all the instances of $N^i$ in diagram \eqref{ss picture 1} with $A\Box_{\Gamma}N(i)$, and the resulting diagram remains a correct picture of the spectral sequence.

This convention for drawing the spectral sequence in \eqref{ss picture 1}, in particular the choice of horizontal and vertical coordinates, is convenient because the bidegrees in the $s$-column are precisely those bidegrees which contribute, in the $E_{\infty}$-page, to $\Cotor^s_{\Gamma}(L,M\otimes_A N)$. %Another reasonable convention with the same desirable property is to draw the spectral sequence as a fourth-quadrant spectral sequence by simply reflecting \eqref{ss picture 1} about the horizontal axis; the result has the differentials going between bidegrees in exactly the same way as Adams spectral sequence differentials, so some readers may prefer that convention. Of course this is a matter of taste.

Theorem \ref{cotor sseq} has corollaries:
\begin{corollary}\label{cotensor kunneth formula}
Let $(A,\Gamma)$ be a bialgebra, and let $L$ be a right $\Gamma$-comodule and $M,N$ left $\Gamma$-comodules. Suppose that $M$ is flat over $A$, and suppose that $N$ has split exhaustive primitive filtration. 
Then the $A$-module $L\Box_{\Gamma}(M\otimes_A N)$ is isomorphic to the sub-$A$-module of $(L\Box_{\Gamma}M) \otimes_A N$ consisting of the elements in the $s=0$-column in the kernel of the $d_r$ differential for every $r\geq 1$.
\end{corollary}
Note that, unlike the Adams spectral sequence (whose convergence properties are discussed in \cite{MR1718076}), there is no issue of conditional convergence in this spectral sequence which could cause $N = \oplus_{n\geq 0}N^n$ in the $E_{\infty}$-page to become $\prod_{n\geq 0} N^n$ in the abutment. This is because the spectral sequence of Theorem \ref{cotor sseq} converges to the colimit, and so the extension problems are organized into a colimit sequence rather than a limit sequence.

\begin{corollary}\label{gysin seq cor}
Let $(A,\Gamma)$ be a bialgebra, and let $L$ be a right $\Gamma$-comodule and $M,N$ left $\Gamma$-comodules. Suppose that $M$ is flat over $A$ and that $N$ has primitive filtration of length $2$, i.e., the quotient $\Gamma$-comodule $N/\left( A\Box_{\Gamma}N\right)$ has trivial $\Gamma$-coaction. Then the spectral sequence of Theorem \ref{cotor sseq} degenerates to a %Gysin-type 
long exact sequence
\noindent
\begin{equation*}%\label{}
\xymatrix{
 && 0 \ar`r_l[ll] `l[dll] [dll]   \\(L\Box_{\Gamma}M)\otimes_A (A\Box_{\Gamma}N) \ar[r] & 
  L\Box_{\Gamma}(M\otimes_AN) \ar[r] &
  (L\Box_{\Gamma}M)\otimes_A \left( N/(A\Box_{\Gamma}N)\right) \ar`r_l[ll] `l[dll] [dll] \\
  \Cotor^1_{\Gamma}(L,M)\otimes_A (A\Box_{\Gamma}N) \ar[r] &
  \Cotor^1_{\Gamma}(L,M\otimes_A N) \ar[r] &
  \Cotor^1_{\Gamma}(L,M)\otimes_A \left( N/(A\Box_{\Gamma}N)\right) \ar`r_l[ll] `l[dll] [dll] \\
  \Cotor^2_{\Gamma}(L,M)\otimes_A (A\Box_{\Gamma}N) \ar[r] &
  \Cotor^2_{\Gamma}(L,M\otimes_A N) \ar[r] &
  \dots .}\end{equation*}
\end{corollary}

\begin{definition}
Given a commutative ring $A$, we will say that an $A$-module $M$ {\em detects nontriviality} if, whenever $N$ is an $A$-module such that $M\otimes_AN \cong 0$, we have $N\cong 0$.
\end{definition}
For example, 
over any commutative ring $A$, nonzero free $A$-modules detect nontriviality. If $A$ is a field, every nonzero $A$-module detects nontriviality. If $A$ is a discrete valuation ring with maximal ideal $\mathfrak{m}$ and fraction field $K$, then $K\oplus A/\mathfrak{m}A$ detects nontriviality. 

One of the most important consequences of Theorem \ref{cotor sseq} is Corollary \ref{converse to al-takhman lemma}:
\begin{corollary}\label{converse to al-takhman lemma}
Let $(A,\Gamma)$ be a bialgebra, and let $L$ be a right $\Gamma$-comodule and $M,N$ left $\Gamma$-comodules. Suppose that the $A$-module $L\Box_{\Gamma}M$ detects nontriviality, suppose that $N$ has split exhaustive primitive filtration, suppose that the canonical map 
\begin{equation}\label{kunneth iso 2}(L\Box_{\Gamma}M)\otimes_A (A\Box_{\Gamma}N) \rightarrow  L\Box_{\Gamma}(M\otimes_AN)\end{equation} is an isomorphism, and suppose that the canonical map 
\begin{equation}\label{kunneth map 1}\Cotor_{\Gamma}^1(L,M)\otimes_A (A\Box_{\Gamma}N) \rightarrow  \Cotor_{\Gamma}^1(L,M\otimes_AN)\end{equation} is injective. 
Then $N$ has trivial coaction.
\end{corollary}
\begin{proof}
The map \eqref{kunneth iso 2} is an isomorphism if and only if all elements in the leftmost column above the bottom row in spectral sequence \eqref{kunneth ss 1} support differentials, and consequently fail to survive to the $E_{\infty}$-page. In particular, if $(L\Box_{\Gamma} M)\otimes_A N^1$ were nonzero, it would have to support a nonzero $d_1$-differential hitting $\Cotor^1_{\Gamma}(L,M)\otimes_A N^0$, and consequently not all elements of $\Cotor^1_{\Gamma}(L,M)\otimes_A N^0$ would survive to the $E_2$-term, much less the $E_{\infty}$-term. Consequently the map \eqref{kunneth map 1} would not be able to be injective. So $(L\Box_{\Gamma} M)\otimes_A N^1$ must be trivial. Since $L\Box_{\Gamma}M$ detects nontriviality, $N^1$ vanishes. Since $N^1 = A\Box_{\Gamma}(N/(A\Box_{\Gamma}N))$, vanishing of $N^1$ ensures that $N(2) \cong N/(A\Box_{\Gamma}N) \cong N(1)$ and consequently $N^2$ also vanishes; by induction, the primitive filtration on $N$ is constant starting at the second term. Exhaustivity of the primitive filtration on $N$ consequently gives us that $A\Box_{\Gamma}N = N$.
\end{proof}

Using Proposition \ref{conn and exhaustive prim filt} and the fact that triviality of $L\Box_{\Gamma}M$ implies triviality of $M$ when $A$ is a field, $(A,\Gamma)$ is a connected graded bialgebra, and $L,M$ are bounded-below graded $\Gamma$-comodules, we have:
\begin{corollary}\label{graded converse to al-takhman lemma}
Let $A$ be a field, and let $\Gamma$ be a connected graded bialgebra over $A$. Let $L$ be a graded right $\Gamma$-comodule, let $M,N$ be graded left $\Gamma$-comodules such that that the canonical map 
\begin{equation}\label{kunneth iso 1a}(L\Box_{\Gamma}M)\otimes_A (A\Box_{\Gamma}N) \rightarrow  L\Box_{\Gamma}(M\otimes_AN)\end{equation} is an isomorphism, suppose that $L,M,$ and $N$ are nonzero and bounded below, and suppose that the canonical map 
\begin{equation}\label{kunneth map 1a}\Cotor_{\Gamma}^1(L,M)\otimes_A (A\Box_{\Gamma}N) \rightarrow  \Cotor_{\Gamma}^1(L,M\otimes_AN)\end{equation} is injective. 
Then $N$ has trivial coaction.
\end{corollary}

Corollaries \ref{converse to al-takhman lemma} and \ref{graded converse to al-takhman lemma} are negative results, and they provide a kind of converse to Corollary \ref{al-takhman cor}: they tell us that, when we have the expected K\"{u}nneth formula \eqref{kunneth iso 1a} for $\Cotor^0$ (i.e., the cotensor product), then we cannot possibly have any reasonable K\"{u}nneth formula describing $\Cotor^1$, since the failure of the canonical map \eqref{kunneth map 1a} to be injective means that $\Cotor_{\Gamma}^1(A,M\otimes_AN)$ cannot decompose as any kind of extension of $\Cotor_{\Gamma}^1(A,M)\otimes_A \Cotor_{\Gamma}^0(A,N)$ by $\Cotor_{\Gamma}^0(A,M)\otimes_A \Cotor_{\Gamma}^1(A,N)$ unless $N$ has trivial coaction. Instead, the best general statement one can make is that one has the spectral sequence of Theorem \ref{cotor sseq}.

One reasonable response to these negative results is to ask under what circumstances we at least have a K\"{u}nneth formula for $\Cotor^0$. We take up this question in \cref{Kunneth formulas for...}.

% f: (A,\Gamma) -> (A,\Sigma) a map of bialgebras.
% Suppose M is a $\Gamma$-subcomodule of $\Gamma$,
%  so A\Box_{\Gamma}M \subseteq A\Box_{\Gamma}\Gamma \cong A.
% Let f^*M be M base-changed to a \Sigma-comodule, i.e.,
%  f^M = M with coaction map the composite
%   M \rightarrow \Gamma\otimes_A M \rightarrow \Sigma\otimes_A M.
% If f is injective, then A\Box_{\Sigma} M \cong A again.
% Now suppose N is another Sigma-comodule.
% Two questions:
% 1. Under what conditions on N is 
%  (A\Box_{\Sigma} M) \otimes_A (A\Box_{\Sigma} N)
%   \rightarrow A\Box_{\Sigma} (M \otimes_A N)
%  an isomorphism (equivalently, surjective)?
% 2. Suppose that everything in sight is graded, and M and N each have 
%  bounded-above comodule primitives. Under what conditions on N does
%  M\otimes_A N have bounded-above comodule primitives?
%
% Perhaps 1a: Under what conditions do d_1-differentials wipe out the 0-column above the 1-line?
%  
%

\section{Example calculations of the K\"{u}nneth spectral sequence for $\Cotor$.}
\label{Example calculations...}

Let $k$ be a field of characteristic $p$ and let $\Gamma$ be the bialgebra $k[\xi]/\xi^p$ with $x$ primitive. Consider the case $L = k = M$ and $N = \Gamma$ of Theorem  \ref{cotor sseq}: the primitive filtration on $N$ is finite, with $N^n$ the $k$-vector space spanned by $\xi^n$ for $0\leq n\leq p-1$ and trivial otherwise. We adopt the notation $k\{ x\}$ for the $k$-vector space with basis $\{x\}$, so that we can make a reasonable drawing of the $E_1$-page of the K\"{u}nneth spectral sequence, here pictured with $p=5$:

\noindent
\begin{figure}[H]
\begin{tikzpicture}[trim left=1cm,xscale=3.8,yscale=0.7]
\draw[->,color=red] (0,1) -- ($(0,1)!0.9!(1,0)$) node{};
\draw[->,color=red] (0,2) -- ($(0,2)!0.9!(1,1)$) node{};
\draw[->,color=red] (0,3) -- ($(0,3)!0.9!(1,2)$) node{};
\draw[->,color=red] (0,4) -- ($(0,4)!0.9!(1,3)$) node{};
\draw[->,color=blue] (1,4) -- ($(1,4)!0.9!(2,0)$) node{};
\draw[->,color=red] (2,1) -- ($(2,1)!0.9!(3,0)$) node{};
\draw[->,color=red] (2,2) -- ($(2,2)!0.9!(3,1)$) node{};
\draw[->,color=red] (2,3) -- ($(2,3)!0.9!(3,2)$) node{};
\draw[->,color=red] (2,4) -- ($(2,4)!0.9!(3,3)$) node{};
\draw[->,color=blue] (3,4) -- ($(3,4)!0.9!(4,0)$) node{};
%\draw[->,color=red] (0,2) -- ($(0,2)!0.9!(1,1)$) node{};
%\draw[->,color=red] (0,3) -- ($(0,3)!0.9!(1,2)$) node{};
%\draw[->,color=red] (1,1) -- ($(1,1)!0.9!(2,0)$) node{};
%\draw[->,color=red] (1,2) -- ($(1,2)!0.9!(2,1)$) node{};
%\draw[->,color=red] (1,3) -- ($(1,3)!0.9!(2,2)$) node{};
%\draw[->,color=red] (2,1) -- ($(2,1)!0.9!(3,0)$) node{};
%\draw[->,color=red] (2,2) -- ($(2,2)!0.9!(3,1)$) node{};
%\draw[->,color=red] (2,3) -- ($(2,3)!0.9!(3,2)$) node{};
%\draw[->,color=orange] (0,2) -- ($(0,2)!0.9!(1,0)$) node{};
%\draw[->,color=orange] (0,3) -- ($(0,3)!0.9!(1,1)$) node{};
%\draw[->,color=orange] (1,2) -- ($(1,2)!0.9!(2,0)$) node{};
%\draw[->,color=orange] (1,3) -- ($(1,3)!0.9!(2,1)$) node{};
%\draw[->,color=orange] (2,2) -- ($(2,2)!0.9!(3,0)$) node{};
%\draw[->,color=orange] (2,3) -- ($(2,3)!0.9!(3,1)$) node{};
%\draw[->,color=blue] (0,3) -- ($(0,3)!0.9!(1,0)$) node{};
%\draw[->,color=blue] (1,3) -- ($(1,3)!0.9!(2,0)$) node{};
%\draw[->,color=blue] (2,3) -- ($(2,3)!0.9!(3,0)$) node{};
\draw (-0.15,4.3) -- (-0.15,-0.35) -- (3.5,-0.35);
\draw (0,0) node{$k\{1\}$};
\draw (1,0) node{$\Cotor^1_{\Gamma}(k,k)\otimes_k k\{1\}$};
\draw (2,0) node{$\Cotor^2_{\Gamma}(k,k)\otimes_k k\{1\}$};
\draw (3,0) node{$\Cotor^3_{\Gamma}(k,k)\otimes_k k\{1\}$};
\draw (4,0) node{$\dots$};
\draw (0,1) node{$k\{\xi\}$};
\draw (1,1) node{$\Cotor^1_{\Gamma}(k,k)\otimes_k k\{\xi\}$};
\draw (2,1) node{$\Cotor^2_{\Gamma}(k,k)\otimes_k k\{\xi\}$};
\draw (3,1) node{$\Cotor^3_{\Gamma}(k,k)\otimes_k k\{\xi\}$};
\draw (4,1) node{$\dots$};
\draw (0,2) node{$k\{\xi^2\}$};
\draw (1,2) node{$\Cotor^1_{\Gamma}(k,k)\otimes_k k\{\xi^2\}$};
\draw (2,2) node{$\Cotor^2_{\Gamma}(k,k)\otimes_k k\{\xi^2\}$};
\draw (3,2) node{$\Cotor^3_{\Gamma}(k,k)\otimes_k k\{\xi^2\}$};
\draw (4,2) node{$\dots$};
\draw (0,3) node{$k\{\xi^3\}$};
\draw (1,3) node{$\Cotor^1_{\Gamma}(k,k)\otimes_k k\{\xi^3\}$};
\draw (2,3) node{$\Cotor^2_{\Gamma}(k,k)\otimes_k k\{\xi^3\}$};
\draw (3,3) node{$\Cotor^3_{\Gamma}(k,k)\otimes_k k\{\xi^3\}$};
\draw (4,3) node{$\dots$};
\draw (0,4) node{$k\{\xi^4\}$};
\draw (1,4) node{$\Cotor^1_{\Gamma}(k,k)\otimes_k k\{\xi^4\}$};
\draw (2,4) node{$\Cotor^2_{\Gamma}(k,k)\otimes_k k\{\xi^4\}$};
\draw (3,4) node{$\Cotor^3_{\Gamma}(k,k)\otimes_k k\{\xi^4\}$};
\draw (4,4) node{$\dots$};
\end{tikzpicture} 
\caption{$E_1^{*,*}\cong \Cotor_{k[\xi]/\xi^p}^*(k,k)\otimes k[\xi]/\xi^p $\\\raggedleft{$\ \ \ \ \ \ \ \ \ \ \Rightarrow \Cotor_{k[\xi]/\xi^p}^*(k,k[\xi]/\xi^p)$}.}
\label{sseq 40999}
\end{figure}
The nonzero differentials are as pictured, but that claim deserves some justification, which we now give, below. 

An alternative construction of the spectral sequence of Theorem \ref{cotor sseq} is given by filtering the cobar complex of $\Gamma$ with coefficients of $N$ by the primitive filtration on $N$. That is, the spectral sequence of Theorem \ref{cotor sseq} is the spectral sequence of the filtered cochain complex
\[ \Gamma^{\otimes_A\bullet}\otimes_A N_0 \subseteq \Gamma^{\otimes_A\bullet}\otimes_A N_1 \subseteq \Gamma^{\otimes_A\bullet}\otimes_A N_2 \subseteq \dots \subseteq \Gamma^{\otimes_A\bullet}\otimes_A N.\]
In the case $\Gamma = k[\xi]/\xi^p$, a cobar complex $1$-cocycle which represents a generator for $\Cotor^1_{\Gamma}(k,k)$ is the primitive $\xi \in k[\xi]/\xi^p$, while a cobar complex $2$-cocycle which represents a generator for $\Cotor^2_{\Gamma}(k,k)$ is the ``transpotent'' \[T\xi := \sum_{i=1}^{p-1}\frac{1}{p}\binom{p}{i}\xi^{p-i}\otimes \xi^i \in k[\xi]/\xi^p\otimes_k k[\xi]/\xi^p\] of $\xi$, and the graded $k$-algebra $\Cotor^*_{\Gamma}(k,k)$ is isomorphic to $\Lambda\left(h\right)\otimes_k k\left[ b\right]$, where $h\in \Cotor^1_{\Gamma}(k,k)$ is represented by $\xi$, and $b\in \Cotor^2_{\Gamma}(k,k)$ is represented by $T\xi$.
The comodule algebra structure on $\Gamma$, and the fact that the primitive filtration on $\Gamma$ is a filtration by comodule ideals, yields that the spectral sequence of Theorem \ref{cotor sseq} is a spectral sequence of $k$-algebras, and so we need only calculate the differentials on $k$-algebra generators for each page. The $E_1$-page is isomorphic to $k[\xi]/\xi^p\otimes_k \Lambda(h) \otimes_k k[b]$, and since $\psi(\xi) = \xi\otimes 1 + 1\otimes \xi$, we have $\delta(\xi) = \xi\otimes 1$ in the cobar complex, and hence the $d_1$-differential $d_1(\xi) = h$. For degree reasons, $d_1(h) = 0$ and $d_1(b) = 0$, so the Leibniz rule gives us that $d_1(\xi^ib^j) = i\xi^{i-1}hb^j$ for all $i<p$ and $d_1(\xi^ihb^j)=0$, yielding the $E_2$-page $E_2^{*,*}\cong \Lambda(h\xi^{p-1})\otimes_k k[b]$. 
The class $h\xi^{p-1}$ is represented by the cobar complex $1$-cochain $\xi\otimes \xi^{p-1}$, and we have $\delta(\xi\otimes \xi^{p-1}) = -\sum_{i=1}^{p-1}\binom{p-1}{i} \xi^i\otimes \xi^{p-1-i}$ in the cobar complex. When we reach the $E_{p-1}$-page of the spectral sequence, we have that $d_{p-1}(h\xi^{p-1})$ is the sum of the terms of $\delta(\xi\otimes \xi^{p-1})$ of primitive filtration\footnote{Remember that we have filtered the cobar complex of $\Gamma$, with coefficients in $\Gamma$, by the primitive filtration on the coefficients.} $p-1$ less than that of $\xi\otimes \xi^{p-1}$. Consequently we have that $d_{p-1}(h\xi^{p-1})$ is the class in the $E_{p-1}$-page represented by the cocycle $-\xi\otimes \xi^{p-1}\otimes 1$, i.e., $d_{p-1}(h\xi^{p-1}) = -(T\xi)\otimes 1$. 
From the Leibniz rule we get that all that remains on the $E_p$-page is the copy of $k$ in bidegree $(0,0)$, so the spectral sequence collapses at that page. This yields the long differentials pictured in Figure \ref{sseq 40999}.

One consequence is that, by taking $p$ to be a large prime, we can get arbitrarily long nonzero differentials in spectral sequence \eqref{kunneth ss 1}.

\section{When do we have a K\"{u}nneth formula for $\Cotor^0$?}
\label{Kunneth formulas for...}

The main result in this section is Theorem \ref{kunneth quot prop}, which give a criterion for the differentials supported on the $s=0$ column of spectral sequence \eqref{kunneth ss 1} to wipe out everything above the $t=0$ row\footnote{Corollary \ref{cotensor kunneth formula} established that this behavior of the differentials in the spectral sequence is equivalent to the natural map $\left(L\Box_{\Gamma}M\right)\otimes_A \left( A\Box_{\Gamma}N\right) \rightarrow L\Box_{\Gamma}(M\otimes_A N)$ being an isomorphism.}, at least in the situation of greatest interest for topological applications, i.e., the situation of Theorem \ref{cotor sseq} when $L = A$ and $M$ is a subcomodule of $\Gamma$. The most obvious cases of those topological applications are described later, in \cref{Topological applications.}. 
In this section we also give Theorem \ref{kunneth quot prop 2}, a generalization of Theorem \ref{kunneth quot prop} which weakens the hypothesis that $M$ is a subcomodule of $\Gamma$.

\begin{definition}\label{def of kunneth quot}
Let $A$ be a commutative ring, let $\Gamma$ be a bialgebra over $A$, and let $M$ be a subcomodule of the left $\Gamma$-comodule $\Gamma$. Let $N$ be a left $\Gamma$-comodule. By the {\em K\"{u}nneth quotient of $N$ relative to $M$} we mean the $A$-module $\Ku(N;M)$ given by the cokernel of the natural $A$-module map 
\[\left( M\otimes_A N^0\right) \times_{\Gamma\otimes_A N} N \hookrightarrow \left( M\otimes_A N\right) \times_{\Gamma\otimes_A N} N\] arising from the commutative diagram
\begin{equation}\label{pullback diag 09530}\xymatrix{
 M\otimes_A N^0 \ar@{^{(}->}[r] & M\otimes_A N \ar@{^{(}->}[r] & \Gamma\otimes_A N\\
 \left( M\otimes_A N^0\right)\times_{\Gamma\otimes_A N}N \ar@{^{(}->}[r] \ar@{^{(}->}[u] & \left(M\otimes_A N\right)\times_{\Gamma\otimes_A N}N \ar@{^{(}->}[r]\ar@{^{(}->}[u] & N\ar@{^{(}->}[u]^{\psi_N}
}\end{equation}
in which each square is a pullback square of $A$-modules.
\end{definition}
The maps marked with hooked arrows in \eqref{pullback diag 09530} are injective if $M$ is flat over $A$.

Of the equivalent conditions given in Theorem \ref{kunneth quot prop}, the fourth condition is an especially checkable (in practical situations) necessary and sufficient condition for the K\"{u}nneth formula for the cotensor product, \eqref{map 340960}, to hold.
\begin{theorem}\label{kunneth quot prop}
Let $A$ be a commutative ring, let $\Gamma$ be a bialgebra over $A$, and let $M$ be a subcomodule of the left $\Gamma$-comodule $\Gamma$. Let $N$ be a left $\Gamma$-comodule which is flat over $A$, and suppose that $M$ is also flat over $A$. Then the following conditions are equivalent:
\begin{enumerate}
\item
The canonical map 
\begin{equation}\label{map 340960} (A\Box_{\Gamma}M)\otimes_A (A\Box_{\Gamma}N) \rightarrow A\Box_{\Gamma}(M\otimes_AN)\end{equation}
is an isomorphism.
\item 
The intersection of $M\otimes_A N \subseteq \Gamma\otimes_A N$ with the image of the coaction map $\psi_N: N \rightarrow \Gamma\otimes_A N$ lands in the submodule $M\otimes_A (A\Box_{\Gamma}N)$ of $M\otimes_A N$. 
\item The K\"{u}nneth quotient $\Ku(N;M)$ vanishes.
%\end{enumerate} If $A$ is also assumed to be a field, then the above conditions are also equivalent to: \begin{enumerate}
\item[(4)] If $n\in N$ satisfies $\psi_N(n)\in M\otimes_A N$, then $n$ is primitive.
\end{enumerate}
\end{theorem}
\begin{proof}\leavevmode
\begin{description}
\item[1 is equivalent to 2] 
Since $M\subseteq \Gamma$ and $M$ is flat over $A$, we have monomorphisms $M\otimes_A N \hookrightarrow \Gamma\otimes_A N$ and 
$A\Box_{\Gamma}(M\otimes_A N) \hookrightarrow A\Box_{\Gamma}(\Gamma\otimes_A N)$, and consequently a commutative diagram of $A$-modules
\[\xymatrix{
 A\Box_{\Gamma}\left(M\otimes_A N\right)\ar@{^{(}->}[r]
  & A\Box_{\Gamma}\left(\Gamma\otimes_A N\right) 
  & N \ar[l]^{\cong}_{sh} \\
 A\Box_{\Gamma}\left(M\otimes_A (A\Box_{\Gamma}N)\right)\ar@{^{(}->}[r]\ar@{^{(}->}[u]
  & A\Box_{\Gamma}\left(\Gamma\otimes_A (A\Box_{\Gamma}N)\right) \ar@{^{(}->}[u] 
  & A\Box_{\Gamma}N \ar[l]^(.37){\cong}_(.37){sh\mid_{A\Box_{\Gamma}N}} \ar@{^{(}->}[u],
}\]
where $sh: N \rightarrow A\Box_{\Gamma}(\Gamma\otimes_A N)$ is the ``shearing isomorphism''%\footnote{Shearing isomorphisms are quite useful, but perhaps not as well-documented in the literature as they could be. Shearing isomorphisms for modules over bialgebras appear in, for example, Definition 2.10 of \cite{MR1375579} and Lemma 1.3.4 of \cite{MR1821838}, and, for a different level of generality, there is a shearing isomorphism for $G$-spaces given preceding Proposition 4.5 in \cite{schwedeequivariantnotes}. We do not know anywhere in the literature where a shearing isomorphism for comodules is explicitly given, but one can deduce that such an isomorphism must exist from Lemma A1.1.6(b) of \cite{MR860042}, and one gets that the formula is given as in \eqref{comm diag 090004} by simply formally dualizing the shearing isomorphism for modules given as in Definition 2.10 of \cite{MR1375579}.} 
given on elements simply by the coaction map $\psi_N$. Expressed more carefully, the diagram 
\begin{equation}\label{comm diag 090004}\xymatrix{
A\Box_{\Gamma}\left(\Gamma\otimes_A N \right)\ar@{^{(}->}[d]^{\iota} && N \ar[ll]^{\cong}_{sh} \ar[lld]^{\psi_N} \\
 \Gamma\otimes_A N &&  
}\end{equation}
commutes, where $\iota$ is the canonical inclusion of the comodule primitives of $\Gamma\otimes_A N$ into $\Gamma\otimes_A N$. 

Consequently any given element of $A\Box_{\Gamma}(M\otimes_AN)\subseteq \Gamma\otimes_A N$ is equal to $\psi(n)$ for some $n\in N$. So the inclusion $A\Box_{\Gamma}(M\otimes_AN^0)\subseteq A\Box_{\Gamma}(M\otimes_AN)$ is surjective if and only if $\psi(n)$ is contained in $M\otimes_A N^0$ whenever $\psi(n)\in M\otimes_A N$, as in the statement of the proposition. 
%\item[1 is equivalent to 3] Since pullbacks of monomorphisms are monomorphisms, the maps marked as hooked arrows in \eqref{pullback diag 09530} are one-to-one. So \eqref{map 340960} is an isomorphism if and only if its cokernel vanishes, i.e., if and only if the K\"{u}nneth quotient vanishes.
\item[2 is equivalent to 3,  and 2 is equivalent to 4] These implications are a routine matter of unwinding the definitions.
\end{description}
\end{proof}

\begin{remark}
It would be nice if the functor $\Ku(-;M):\Comod(\Gamma)\rightarrow\Mod(A)$ were at least half-exact, as that would give us some means (by standard homological methods) of actually computing its value on various comodules. 
However, $\Ku$ is generally not half-exact, as one sees from the example where $A$ is a field $k$, and $M = \Gamma = k[\xi]/\xi^2$ with $\xi$ primitive. If we had any flexible and powerful tools for computing $\Ku(-;M)$, then the vanishing of $\Ku(N;M)$ could potentially be the most checkable of the four equivalent conditions listed in Theorem \ref{kunneth quot prop}, instead of (as presently seems to be the case) the fourth condition being the most practically checkable. Unfortunately, the author does not know of any such tools for calculating $\Ku$, and would be glad to learn of them. 

If $\Gamma$ is finitely generated as an $A$-module (or if $(A,\Gamma)$ is graded with $A$ concentrated in degree $0$ and $\Gamma$ finitely generated as an $A$-module in each degree) then $\Ku$ is at least a {\em coherent} functor in the sense of \cite{MR0212070}, but this does not seem to yield any nontrivial means of computing the values of $\Ku(-;M)$. 
\end{remark}

Theorem \ref{kunneth quot prop} admits a generalization to situations where $M$ is not a subcomodule of $\Gamma$. That generalization is Theorem \ref{kunneth quot prop 2}, which is not as clean to state as Theorem \ref{kunneth quot prop}, which is why we present the result separately. The material on the K\"{u}nneth quotient $\Ku(N;M)$ from \ref{kunneth quot prop} also has a generalization to the setting of Theorem \ref{kunneth quot prop 2}, but since the vanishing of the K\"{u}nneth quotient is the condition, among the equivalent conditions of Theorem \ref{kunneth quot prop}, which we know the least applications for, we leave off the K\"{u}nneth quotient from Theorem \ref{kunneth quot prop 2} to try to keep the statement cleaner.
\begin{theorem}\label{kunneth quot prop 2}
Let $(A,\Gamma)$ be a graded bialgebra with $A$ concentrated in degree zero. Suppose that $\Gamma$ is finite-type, that is, for each integer $n$ the degree $n$ summand of $\Gamma$ is a finitely generated $A$-module.
Let $M$ be a graded left subcomodule of a finite-type direct sum of suspensions of $\Gamma$, that is, there exists some function $d: \mathbb{Z}\rightarrow \mathbb{N}$ and some one-to-one graded left $\Gamma$-comodule homomorphism
\[ \iota: M \hookrightarrow \coprod_{n\in \mathbb{Z}}\Sigma^{n}\Gamma^{\oplus d(n)} .\]
Suppose that $N$ is a graded right $\Gamma$-comodule, and suppose that $M$ and $N$ are each flat over $A$.
Then the following conditions are equivalent:
\begin{enumerate}
\item
The canonical map 
\begin{equation}\label{map 340960a} (A\Box_{\Gamma}M)\otimes_A (A\Box_{\Gamma}N) \rightarrow A\Box_{\Gamma}(M\otimes_AN)\end{equation}
is an isomorphism.
\item 
The intersection of the image of \[\iota\otimes N: M\otimes_A N \hookrightarrow\left( \coprod_{n\in \mathbb{Z}}\Sigma^{n}\Gamma^{\oplus d(n)}\right)\otimes_A N\] with the image of the direct sum of copies of the coaction map \[\coprod_{n\in\mathbb{Z}}\Sigma^{n}\psi_N^{\oplus d(n)}: \coprod_{n\in\mathbb{Z}}\Sigma^{n}N^{\oplus d(n)} \rightarrow \coprod_{n\in \mathbb{Z}}\Sigma^{n}\left(\Gamma\otimes_A N\right)^{\oplus d(n)} \stackrel{\cong}{\longrightarrow}\left(\coprod_{n\in\mathbb{Z}}\Sigma^{n}\Gamma\right)^{\oplus d(n)}\otimes_A N\] lands in the submodule $M\otimes_A (A\Box_{\Gamma}N)$ of $M\otimes_A N$. 
%\end{enumerate} If $A$ is also assumed to be a field, then the above conditions are also equivalent to: \begin{enumerate}
\item[(3)] If $y\in \coprod_{n\in\mathbb{Z}} \Sigma^n N^{\oplus d(n)}$ satisfies \[\coprod_{n\in\mathbb{Z}}\Sigma^n\psi_N^{\oplus d(n)}(y)\in M\otimes_A N\subseteq \coprod_{n\in\mathbb{Z}} \Sigma^n \Gamma^{\oplus d(n)}\otimes_{A} N,\] then $y$ is a primitive element of $\coprod_{n\in\mathbb{Z}} \Sigma^n N^{\oplus d(n)}$.
\end{enumerate}
\end{theorem}
\begin{proof}
The proof is essentially the same as that of Theorem \ref{kunneth quot prop}.
\end{proof}

%\begin{remark} Perhaps the finiteness hypotheses in Theorem \ref{kunneth quot prop 2} can be weakened. The only way that the finiteness hypotheses are used in the proof of the theorem is so that we can freely commute coproducts with cotensor products. We do not know of any counterexamples to the statement of the theorem with the finiteness hypotheses omitted, but then again, we have not looked for such counterexamples. \end{remark}

%The reason for the sudden introduction of gradings in Theorem \ref{kunneth quot prop 2} is so that we can consider K\"{u}nneth isomorphisms in $\Cotor$ over familiar graded bialgebras from topology which are finite-type but not finite-dimensional. 
The most notable family of examples of graded bialgebras satisfying the hypotheses of Theorem \ref{kunneth quot prop 2} are the mod $p$ Steenrod algebras and their linear duals at each prime $p$. The associated graded bialgebras $E^0S(n)$ of Ravenel's grading on the Morava stabilizer algebras (as in \cite{ravenel1977cohomology} and section 6.3 of \cite{MR860042}) form another important family of examples of finite-type, non-finite-dimensional graded bialgebras whose $\Cotor$ groups are, like $\Cotor$ over the duals of the Steenrod algebras, the input for spectral sequences which ultimately compute stable homotopy groups of various spaces and spectra.

However, the most straightforward case of Theorem \ref{kunneth quot prop 2} is, of course, the case where the rings and modules in question are concentrated in degree $0$, i.e., the ungraded case:
\begin{corollary}\label{kunneth quot prop 2 cor 1}
Let $A$ be a commutative ring, and let $\Gamma$ be a bialgebra over $A$. Suppose that $\Gamma$ is finitely generated as an $A$-module. 
Let $M$ be a left subcomodule of a direct sum $\Gamma^{\oplus n}$ of finitely many copies of $\Gamma$, let $N$ be a right $\Gamma$-comodule, and suppose that $M,N$ are each flat over $A$.
Then the following conditions are equivalent:
\begin{enumerate}
\item
The canonical map 
\begin{equation*}%\label{map 340960a} 
(A\Box_{\Gamma}M)\otimes_A (A\Box_{\Gamma}N) \rightarrow A\Box_{\Gamma}(M\otimes_AN)\end{equation*}
is an isomorphism.
\item 
The intersection of the image of $\iota\otimes N: M\otimes_A N \hookrightarrow\left( \Gamma^{\oplus n}\right)\otimes_A N$ with the image of the direct sum of copies of the coaction map \[\psi_N^{\oplus n}: N^{\oplus n} \rightarrow \left(\Gamma\otimes_A N\right)^{\oplus n} \stackrel{\cong}{\longrightarrow} \Gamma^{\oplus n}\otimes_A N\] lands in the submodule $M\otimes_A (A\Box_{\Gamma}N)$ of $M\otimes_A N$. 
%\end{enumerate} If $A$ is also assumed to be a field, then the above conditions are also equivalent to: \begin{enumerate}
\item[(3)] If $n\in N^{\oplus n}$ satisfies $\psi_N(n)\in M\otimes_A N\subseteq \Gamma^{\oplus n}\otimes_{A} N$, then $n$ is a primitive element of $N^{\oplus n}$.
\end{enumerate}
\end{corollary}

\begin{example}\leavevmode
\begin{itemize}
\item
See the calculations of \cref{Example calculations...} for a case in which the the isomorphism \ref{map 340960a} holds, i.e., a case in which we have a K\"{u}nneth formula for $\Cotor^0$.
\item
On the other hand, let $A=k$ for some field $k$ of characteristic $p$, and let $\Gamma = k[\xi]/\xi^p$ with $\xi$ primitive. Then the canonical map 
\[ (A\Box_{\Gamma}\Gamma)\otimes_A (A\Box_{\Gamma}\Gamma) \rightarrow A\Box_{\Gamma}(\Gamma\otimes_A\Gamma) \]
is not surjective: its domain is isomorphic to $k$, while its codomain is isomorphic to $\Gamma$. (Of course, this same example works in the same way for any bialgebra over a field, but we chose $\Gamma$ as above for concreteness and consistency with \cref{Example calculations...}.)
\end{itemize}
\end{example}

\section{Topological applications.}
\label{Topological applications.}

Now we had better explain some of the topological consequences of the results obtained in this paper.
Given a pointed space or spectrum $X$ and an integer $n$, it is classical that we can inductively attach cells to $X$ to wipe out all the homotopy groups of $X$ in degrees greater than $n$, yielding a pointed space or spectrum $X^{\leq n}$ and a continuous map $X\rightarrow X^{\leq n}$ such that $\pi_m(X) \rightarrow \pi_m(X^{\leq n})$ is an isomorphism for all $m\leq n$, and such that $\pi_m(X^{\leq n})$ vanishes for all $m>n$. In the unstable setting, this construction dates back to \cite{MR46045}, solving a problem of Hurewicz's from Eilenberg's 1949 list \cite{MR30189} of open problems in algebraic topology.

When we take the homotopy fiber of $X\rightarrow X^{\leq n}$, we get a homotopy fiber sequence 
\begin{equation}\label{htpy fib seq 1} X^{>n} \rightarrow X\rightarrow X^{\leq n},\end{equation} where $X^{>n} \rightarrow X$ induces an isomorphism in $\pi_m$ for all $m>n$, and $\pi_m(X^{>n})$ vanishes for $m\leq n$. If we take this homotopy fiber in the stable homotopy category, then \eqref{htpy fib seq 1} is also a homotopy {\em co}fiber sequence, and so it induces a long exact sequence in homology groups. It is natural to ask under what conditions on $X$ and $n$ we can get good computational control over that long exact sequence in homology. 

In particular, %in the preprint \cite{Kthyofholims}, it is shown that 
the homotopy cofiber sequence \eqref{htpy fib seq 1} induces a {\em short} exact sequence of graded $\mathbb{F}_p$-vector spaces
\begin{equation}\label{Ses 000349} 0 \rightarrow H_*(X; \mathbb{F}_p) \rightarrow H_*(X^{\leq n}; \mathbb{F}_p)  \rightarrow H_*(\Sigma X^{>n}; \mathbb{F}_p) \rightarrow 0 \end{equation}
if we assume that $X$ is bounded below, $H\mathbb{F}_p$-nilpotently complete\footnote{The standard reference for nilpotent completion of spectra is \cite{MR551009}. We offer a bit of explanation to make the idea concrete, for readers not already familar with nilpotent completion of spectra: under the assumption that a spectrum $X$ is bounded below, Bousfield proves that $X$ is $H\mathbb{F}_p$-nilpotently complete if and only if its homotopy groups are $\Ext$-$p$-complete in the sense of \cite{MR0365573}. More concretely, if $X$ is bounded-below and each of its homotopy groups are finitely generated, then $X$ is $H\mathbb{F}_p$-nilpotently complete if and only if each of its homotopy groups are $p$-adically complete.}, and the $A_*$-comodule primitives of $H^*(X; \mathbb{F}_p)$ are trivial in degrees $\geq n$. Here $p$ is any prime, and  $A_*$ is the mod $p$ dual Steenrod algebra. 

In other words, if $\Cotor^0_{A_*}\left(\mathbb{F}_p,H_*(X; \mathbb{F}_p)\right)$ is bounded above (i.e., vanishes in sufficiently large degrees), then \eqref{Ses 000349} is short exact for sufficently large $n$.
%\end{comment}
%\end{enumerate}

%One conclusion is that $\Cotor$ over coalgebras, especially $\Cotor$ over the dual Steenrod algebra, is very useful in certain topological applications. Since the classical K\"{u}nneth formula over a field takes the form of an isomorphism $H_*(X\times Y; k)\cong H_*(X;k)\otimes_k H_*(Y; k)$ of graded $A_*$-comodules for any spaces $X,Y$, and an isomorphism $H_*(X\smash Y; k)\cong H_*(X;k)\otimes_k H_*(Y; k)$ of graded $A_*$-comodules for any spectra $X,Y$, a formula for $\Cotor^*_{\Gamma}\left( k, M\otimes_k N\right)$ would be quite useful in both unstable settings (via the Eilenberg-Moore spectral sequence) and stable settings (via the Adams spectral sequence).

The $0$-line in the Adams $E_2$-page for a smash product $X\smash Y$ is \[ \Cotor^0_{A_*}\left(\mathbb{F}_p, H_*(X;\mathbb{F}_p)\otimes_{\mathbb{F}_p} H_*(Y; \mathbb{F}_p)\right).\] The $0$-line in the Adams spectral sequence is of particular interest, since the $0$-line in the $E_{\infty}$-page is precisely the image of the Hurewicz homomorphism from stable homotopy to homology. 
In light of the above considerations about attaching cells to kill higher homotopy, if we know that $\Cotor^0_{A_*}\left(\mathbb{F}_p, H_*(X;\mathbb{F}_p)\right)$ and $\Cotor^0_{A_*}\left( \mathbb{F}_p, H_*(Y; \mathbb{F}_p)\right)$ are each bounded above, we would like to know that $\Cotor^0_{A_*}\left(\mathbb{F}_p, H_*(X;\mathbb{F}_p)\otimes_{\mathbb{F}_p} H_*(Y;\mathbb{F}_p)\right)$ is also bounded above, so that 
\[ 
 0 
  \rightarrow H_*\left( X\smash Y; \mathbb{F}_p\right)
  \rightarrow H_*\left( (X\smash Y)^{<n}; \mathbb{F}_p\right)
  \rightarrow H_*\left( \Sigma (X\smash Y)^{\geq n}; \mathbb{F}_p\right)
  \rightarrow 0\]
is short exact for some $n$.

Of course $\Cotor^0_{A_*}(\mathbb{F}_p,-)$ is simply the $A_*$-comodule primitives functor, so as a special case of Theorem \ref{kunneth quot prop}, we have:
\begin{corollary}\label{main adams cor 1}
Let $X,Y$ be spectra, and suppose that $H_*(X;\mathbb{F}_p)$ is a $A_*$-subcomodule of $A_*$. Let $e(X),e(Y),e(X\smash Y)$ denote the $0$-line in the Adams $E_2$-term for $X$, $Y$, and $X\smash Y$, respectively. Then the canonical map $e(X)\otimes_{\mathbb{F}_p}e(Y) \rightarrow e(X\smash Y)$ is an isomorphism if and only if the following condition is satisfied:
\begin{quote}
\label{condition 1}
For all homogeneous $n \in H_*(Y; \mathbb{F}_p)$ such that $\psi(n)\in H_*(X;\mathbb{F}_p)\otimes_{\mathbb{F}_p} H_*(Y; \mathbb{F}_p) \subseteq A_*\otimes_{\mathbb{F}_p}H_*(Y; \mathbb{F}_p)$, we have that $n$ is an $A_*$-comodule primitive.\end{quote}

In particular, if the $A_*$-comodule primitives in $H_*(X; \mathbb{F}_p)$ and $H_*(Y; \mathbb{F}_p)$ are each bounded above, and condition \eqref{condition 1} is satisfied, then the $A_*$-comodule primitives in $H_*(X; \mathbb{F}_p)$ are also bounded above.
\end{corollary}
There are many familiar and compelling examples of spectra $X$ such that $H_*(X;\mathbb{F}_p)$ is a $A_*$-subcomodule of $A_*$: for example, the case $X = BP$, or the case $X = BP\langle n\rangle$ for any positive integer $n$, or the cases $X = ko$ or $X = tmf$ when $p=2$. In these cases, we have:
\begin{corollary}\label{main adams cor 1a}
Let $Y$ be a spectrum. We continue to write $e(X)$ for the $0$-line in the Adams $E_2$-spectral sequence for a spectrum $X$. We write $\zeta_n$ for the conjugate $\overline{\xi}_n$ of $\xi_n$ in the dual Steenrod algebra $A_*$.
\begin{itemize}
\item Let $p=2$. 
Then the canonical map $e(BP)\otimes_{\mathbb{F}_2}e(Y) \rightarrow e(BP\smash Y)$ is an isomorphism if and only if, for each homogeneous element $y\in H_*(Y; \mathbb{F}_2)$ such that $\psi(y) \in P(\zeta_1^2, \zeta_2^2, \dots )\otimes_{\mathbb{F}_2} H_*(Y; \mathbb{F}_2)$, we have $\psi(y) = 1\otimes y$.
\item Let $p>2$. 
Then the canonical map $e(BP)\otimes_{\mathbb{F}_p}e(Y) \rightarrow e(BP\smash Y)$ is an isomorphism if and only if, for each homogeneous element $y\in H_*(Y; \mathbb{F}_p)$ such that $\psi(y) \in P(\xi_1, \xi_2, \dots )\otimes_{\mathbb{F}_p} H_*(Y; \mathbb{F}_p)$, we have $\psi(y) = 1\otimes y$.
\item Let $p=2$.
The canonical map $e(BP\langle n\rangle)\otimes_{\mathbb{F}_2}e(Y) \rightarrow e(BP\langle n\rangle\smash Y)$ is an isomorphism if and only if, for each homogeneous element $y\in H_*(Y; \mathbb{F}_2)$ such that $\psi(y) \in P(\zeta^2_1, \dots, \zeta^2_n,\zeta_{n+1},\zeta_{n+2}, \dots )\otimes_{\mathbb{F}_2} H_*(Y; \mathbb{F}_2)$, we have $\psi(y) = 1\otimes y$.
\item Let $p>2$.
The canonical map $e(BP\langle n\rangle)\otimes_{\mathbb{F}_p}e(Y) \rightarrow e(BP\langle n\rangle\smash Y)$ is an isomorphism if and only if, for each homogeneous element $y\in H_*(Y; \mathbb{F}_p)$ such that $\psi(y) \in P(\xi_{1}, \xi_{2}, \dots )\otimes_{\mathbb{F}_p}E(\tau_n,\tau_{n+1}, \dots)\otimes_{\mathbb{F}_p} H_*(Y; \mathbb{F}_p)$, we have $\psi(y) = 1\otimes y$.
\item Let $p=2$.
The canonical map $e(ko)\otimes_{\mathbb{F}_2}e(Y) \rightarrow e(ko\smash Y)$ is an isomorphism if and only if, for each homogeneous element $y\in H_*(Y; \mathbb{F}_2)$ such that $\psi(y) \in P(\zeta^4_1, \zeta^2_2, \zeta_3,\zeta_4, \dots )\otimes_{\mathbb{F}_2} H_*(Y; \mathbb{F}_2)$, we have $\psi(y) = 1\otimes y$.
\item Let $p=2$.
The canonical map $e(tmf)\otimes_{\mathbb{F}_2}e(Y) \rightarrow e(tmf\smash Y)$ is an isomorphism if and only if, for each homogeneous element $y\in H_*(Y; \mathbb{F}_2)$ such that $\psi(y) \in P(\zeta^8_1, \zeta^4_2, \zeta^2_3,\zeta_4, \zeta_5, \dots )\otimes_{\mathbb{F}_2} H_*(Y; \mathbb{F}_2)$, we have $\psi(y) = 1\otimes y$.
\end{itemize}
\end{corollary}
The cases of Corollary \ref{main adams cor 1a} can be handled by an alternate method, since the spectra $BP$ and $BP\langle n\rangle$ have the property that their mod $p$ homology is not only an $A_*$-subcomodule of $A_*$ but also a sub-{\em bialgebra} of $A_*$. The same is true of $ko$ and $tmf$ at the prime $2$. When $H_*(X;\mathbb{F}_p)$ is a sub-bialgebra of $A_*$, we can use a change-of-rings isomorphism 
\begin{align*}
 \Cotor^n_{A_*}\left(\mathbb{F}_p,H_*(X\smash Y; \mathbb{F}_p)\right)
  &\cong \Cotor^n_{A_*}\left(\mathbb{F}_p,H_*(X; \mathbb{F}_p)\otimes_{\mathbb{F}_p} H_*(Y; \mathbb{F}_p)\right) \\
  &\cong \Cotor^n_{H_*(X; \mathbb{F}_p)}\left(\mathbb{F}_p,H_*(Y; \mathbb{F}_p)\right)  %ggxx
\end{align*}
to arrive at the same conclusions as in Corollary \ref{main adams cor 1a}.
The advantage of using Corollary \ref{main adams cor 1} rather than change-of-rings isomorphisms is flexibility and generality: Corollary \ref{main adams cor 1} does not require $H_*(X; \mathbb{F}_p)$ to have its own comultiplication. 

More broadly, the methods of this paper, particularly Theorem \ref{kunneth quot prop 2}, also apply when $H_*(X; \mathbb{F}_p)$ is not even a subcomodule of $A_*$, much less a subcoalgebra. This includes some familiar classical cases: for example, $X = ku$ or $X = MU$ at odd primes. 
As a special case of Theorem \ref{kunneth quot prop 2}, 
we have the generalization of Corollary \ref{main adams cor 1}:
\begin{corollary}\label{main adams cor 2}
Let $X,Y$ be spectra, and suppose that $H_*(X;\mathbb{F}_p)$ is a graded sub-$A_*$-comodule of $\coprod_{n\in\mathbb{Z}}\Sigma^n A_*^{\oplus d(n)}$ for some function $d:\mathbb{Z}\rightarrow\mathbb{N}$. Then $e(X)\otimes_{\mathbb{F}_p}e(Y)\rightarrow e(X\smash Y)$ is an isomorphism if and only if, for each homogeneous element $y\in H_*(Y; \mathbb{F}_p)$ such that 
\[ \coprod_{n\in\mathbb{Z}}\Sigma^n\psi^{\oplus d(n)}(y) \in H_*(X; \mathbb{F}_p)\otimes_{\mathbb{F}_p} H_*(Y; \mathbb{F}_p) \subseteq \coprod_{n\in\mathbb{Z}} \Sigma^nA_*^{\oplus d(n)}\otimes_{\mathbb{F}_p} H_*(Y; \mathbb{F}_p),\]
we have $\coprod_{n\in\mathbb{Z}}\Sigma^n\psi^{\oplus d(n)}(y) = \coprod_{n\in\mathbb{Z}}\Sigma^n(1\otimes y)^{\oplus d(n)}$, i.e., $y$ is a $A_*$-comodule primitive.
\end{corollary}

A very familiar example of a spectrum to which Corollary \ref{main adams cor 2} applies is the case $X = ku$:
\begin{corollary}\label{main adams cor 2a}
Let $p>2$.
The canonical map $e(ku)\otimes_{\mathbb{F}_p}e(Y) \rightarrow e(ku\smash Y)$ is an isomorphism if and only if, for each homogeneous element $y\in H_*(Y; \mathbb{F}_p)$ such that 
\[ \coprod_{i=0}^{p-2}\Sigma^{2i}\psi(y) \in \coprod_{i=0}^{p-2}\Sigma^{2i}P(\xi_{1}, \xi_{2}, \dots )\otimes_{\mathbb{F}_p}E(\tau_2,\tau_{3}, \dots)\otimes_{\mathbb{F}_p} H_*(Y; \mathbb{F}_p)
\subseteq \coprod_{i=0}^{p-2}\Sigma^{2i}A_*\otimes_{\mathbb{F}_p}H_*(Y;\mathbb{F}_p),\]
we have $\psi(y) = 1\otimes y$.
\end{corollary}

\bibliography{/home/asalch/texmf/tex/salch}{}
\bibliographystyle{plain}
\end{document}